\documentclass[12pt]{amsart}
\usepackage{amsmath,latexsym,amssymb}
\usepackage{enumerate}
\usepackage{amsthm}
\usepackage{epsfig,graphicx}
\usepackage{tikz}
\usetikzlibrary{arrows,automata}
\usepackage{mathrsfs}
\usepackage{color}

\numberwithin{equation}{section}

\newtheorem{theorem}{Theorem}[section]

\newtheorem{corollary}[theorem]{Corollary}

\newtheorem{lemma}[theorem]{Lemma}
\newtheorem{remark}[theorem]{Remark}
\newtheorem*{remarks*}{Remarks}
\newtheorem{proposition}[theorem]{Proposition}

\newcommand\NN{\mathbb{N}}
\newcommand\RR{\mathbb{R}}

\newcommand\EE{\mathbb{E}}
\newcommand\PP{\mathbb{P}}
 
\newcommand\CB{{\mathcal B}}

\newcommand\CL{{\mathcal L}}

\newcommand\CP{{\mathcal P}}

\newcommand\pp{\mathbf{p}}
\newcommand\ii{\mathbf{i}}

\newcommand\xx{\mathbf{x}}
\newcommand\ww{\mathbf{w}}
\newcommand\eps{\varepsilon}

\title[Random subsets of self-similar sets]{The dimension of random subsets of self-similar sets generated by branching random walk}

\author[P. Allaart]{Pieter Allaart}
\address[P. Allaart]{Mathematics Department, University of North Texas, 1155 Union Cir \#311430, Denton, TX 76203-5017, U.S.A.}
\email{allaart@unt.edu}

\author[L. Streck]{Lauritz Streck}
\address[L. Streck]{School of Mathematics,
University of Edinburgh,
James Clerk Maxwell Building,
Peter Guthrie Tait Road,
Edinburgh,
EH9 3FD, UK.}
\email{lstreck@ed.ac.uk}

\begin{document}

\subjclass[2020]{Primary: 28A78, 28A80}

\begin{abstract}
Given a self-similar set $\Lambda$ that is the attractor of an iterated function system (IFS) $\{f_1,\dots,f_N\}$, consider the following method for constructing a random subset of $\Lambda$: Let $\mathbf{p}=(p_1,\dots,p_N)$ be a probability vector, and label all edges of a full $M$-ary tree independently at random with a number from $\{1,2,\dots,N\}$ according to $\mathbf{p}$, where $M\geq 2$ is an arbitrary integer. Then each infinite path in the tree starting from the root receives a random label sequence which is the coding of a point in $\Lambda$. We let $F\subset\Lambda$ denote the set of all points obtained in this way. This construction was introduced by Allaart and Jones [{\em J. Fractal Geom.} 12 (2025), 67--92], who considered the case of a homogeneous IFS on $\mathbb{R}$ satisfying the Open Set Condition (OSC) and proved non-trivial upper and lower bounds for the Hausdorff dimension of $F$. We demonstrate that under the OSC, the Hausdorff (and box-counting) dimension of $F$ is equal to the upper bound of Allaart and Jones, and extend the result to higher dimensions as well as to non-homogeneous self-similar sets.
\end{abstract}

\keywords{Self-similar set, Random fractal, Hausdorff dimension, Branching random walk}

\maketitle

\section{Introduction}

In this paper we compute the Hausdorff dimension of certain random subsets of self-similar sets in $\RR^d$, constructed via a branching random walk, or equivalently, via random labelings of infinite trees. Our main theorem solves a problem introduced by Allaart and Jones \cite{Allaart-Jones} in a considerably more general setting.

A \textit{self-similar set} in $\RR^d$, corresponding to a finite collection of similar contractions $f_1, \dots, f_N \colon \RR^d \to \RR^d$, is the unique non-empty compact set $\Lambda$ such that 
\[
\Lambda=\bigcup_{i=1}^N f_i(\Lambda).
\]
It is said to satisfy the {\em open set condition} (OSC) if there is a non-empty bounded open set $U \subset \RR^d$ such that $f_i(U) \subset U$ for all $i$ and $f_i(U) \cap f_j(U)= \emptyset$ for all $i \neq j$.  While computing the dimensions of self-similar sets is a difficult problem in general, it has long been known that under the open set condition, the Hausdorff and box-counting dimensions of $\Lambda$ are equal to the unique solution $s$ of the equation
\[
\sum_{i=1}^N r_i^s=1,
\]
where $r_i$ is the contraction ratio, or Lipschitz constant, of $f_i$; see \cite{Hutchinson,Moran}.

Throughout this paper,  we assume that we are given a self-similar set $\Lambda$ satisfying the OSC. (For  technical reasons we also assume that the above $U$ is the interior of its closure- this will be explained in Section \ref{sec:result} in more detail).  The set $\Lambda$ can then be written as
\[
\Lambda=\left\{ x=\lim_{n \to \infty} f_{i_1} \circ \dots \circ f_{i_n}(0): \mathbf{i} \in \{1,  \dots,  N\}^{\NN}\right\}.
\]

We now construct a random subset $F$ of $\Lambda$ as follows.  Fix an $M \in \mathbb{N}$ with $M\geq 2$ and consider the infinite $M$-ary tree (the infinite tree in which the root vertex has degree $M$ and all other vertices have degree $M+1$) with edge set $E$.  Fix a probability vector $\pp=(p_1, \dots, p_N) \in (0, 1)^N$ and let $\omega$ be the random vector on $\{1, \dots, N\}^E$ such that each $\omega(e)$ is distributed according to $\pp$ (so $\PP(\omega(e)=j)=p_j$ for all $j$) and $\omega(e), e\in E$ are independent.  We let $\mathcal{P}$ be the set of infinite paths emanating from the root and write $P(k)$ for the $k$-th edge of a path $P \in \mathcal{P}$.  In this paper, we calculate the dimension of the random subset $F$ of $\Lambda$ defined by
\begin{equation} \label{eq:random-subset}
F=F(M, \pp):=\left\{ x=\lim_{n \to \infty} f_{\omega(P(1))} \circ \dots \circ f_{\omega(P(n))}(0) \Big| P \in \mathcal{P} \right\}.
\end{equation}

To help the reader visualize the model, we first consider the explicit example of the Cantor set in $\RR$ (corresponding to $f_1(x)=x/3$ and $f_2(x)=(x+2)/3$) and the binary tree.  In this setting, we are given the Cantor set
\[
C=\left\{x=2\sum_{i=1}^\infty \frac{x_i}{3^i}: x_i\in\{0,1\}\ \forall i\right\},
\]
a number $p\in(0,1)$, and the infinite binary tree emanating from a root vertex in which each edge $e\in E$ has an associated Bernoulli random variable $\omega(e)$ with $\PP(\omega(e)=0)=p,  \PP(\omega(e)=1)=1-p$.  The random set $F$ from \eqref{eq:random-subset} is then given by
\[
F=\left\{x=2\sum_{i=1}^\infty \frac{\omega(P(i))}{3^i} : P \in \mathcal{P} \right\}.
\]
 
Or, to describe the construction of $F$ more dynamically: Imagine a ``branching random walk with exponentially decreasing step sizes", as follows. At time $n=0$, an ancestor particle sits at $x=0$. Just before time $n=1$, this particle splits into two ``children", each of which instantaneously and independently either stays at $x=0$ with probability $p$, or moves to the right by $2/3$ with probability $1-p$. This process continues: At any time $n$, there are $2^n$ particles, each of which splits into two particles just before time $n+1$. The $2^{n+1}$ children each either stay at the same location as their parent particle, or move to the right by $2(1/3)^{n+1}$, independently of each other and of previous generations. Thus, each particle always moves to the left endpoint of a ``basic interval" of the Cantor set. Let $F_n$ denote the set of all points occupied by at least one particle at time $n$. Thus, $F_0=\{0\}$, 
\[
F_1=\begin{cases}
\{0\} & \mbox{with probability $p^2$},\\
\{2/3\} & \mbox{with probability $(1-p)^2$},\\
\{0,2/3\} & \mbox{with probability $2p(1-p)$},
\end{cases}
\]
etc. The sequence of subsets $(F_n)$ then converges in the Hausdorff metric to the set $F$ we are interested in.  

Before stating the most general version of our results (Theorem \ref{thm:main} below), we first present two special cases, both with two maps $f_1$ and $f_2$ and a binary tree (so $N=M=2$). We assume the OSC is satisfied, with the open set $U$ being the interior of its closure. As above, we are given a $p\in(0,1)$ such that $\PP(\omega(e)=0)=p,  \PP(\omega(e)=1)=1-p$ for all $e\in E$. The first theorem gives an explicit formula for the dimension of $F$ in the homogeneous case, when $r_1=r_2$. The second theorem deals with the inhomogeneous case ($r_1\neq r_2$) and illustrates that there are in general three essentially different cases with regard to the probability vector $\pp$. This sets our model apart from other popular random fractal constructions, such as Mandelbrot percolation, where (conditioned on non-extinction) the Hausdorff dimension is given by a single formula.


\begin{theorem} \label{cor:two-maps-homogeneous}
For the homogeneous IFS $\{f_1,f_2\}$ with equal ratios $r_1=r_2=r$ and the binary tree,  
$$\dim_H F=\dim_B F=\frac{\xi\log\xi+(1-\xi)\log(1-\xi)}{\log r}$$
almost surely,  where $\xi=\xi(p)$ is defined by
$$\xi(p):=\frac{\log 2p}{\log p-\log(1-p)}$$
for $p \neq 1/2$, and $\xi(1/2):=1/2$.
\end{theorem}


The setting of Theorem \ref{cor:two-maps-homogeneous} (and arbitrary homogeneous self-similar sets) was recently considered by Allaart and Jones in \cite{Allaart-Jones}, who established the upper bound in Corollary \ref{cor:homogeneous} below and gave a non-trivial lower bound (which is sharp in the case of $p=1/2$). In that paper it is also shown that, when $p=1/2$,  the random subset $F$ has the maximal Hausdorff dimension $s$ but has $s$-dimensional Hausdorff measure zero. This article improves and significantly extends the results of \cite{Allaart-Jones}, on the one hand by calculating the exact value for the dimension and on the other by proving the results in a much more general setting (of inhomogeneous self-similar sets with arbitrarily many maps and a full $M$-ary trees, for arbitrary $M\geq 2$). 

Suppose next that we have two similarities $f_1, f_2: \RR^d\to\RR^d$ but with different contraction ratios $r_1<r_2$, which makes the dimension of $\Lambda$ equal to the unique number $s_0$ such that 
\begin{equation} \label{eq:Moran-2-maps}
r_1^{s_0}+r_2^{s_0}=1.
\end{equation}

\begin{theorem} \label{cor:two-maps-nonhomogeneous}
For the IFS $\{f_1,f_2\}$ with ratios $r_1<r_2$ and the binary tree, let $s_0$ be given by \eqref{eq:Moran-2-maps}, and let $\tilde{s}=\tilde{s}(p)$ satisfy 
\begin{equation} \label{eq:s-tilde-equation}
pr_1^{\tilde{s}}+(1-p)r_2^{\tilde{s}}=\frac12.
\end{equation}
Furthermore, define
\[
\hat{s}=\hat{s}(p):=\frac{\log c_1(p)\log(2(1-p))-\log c_2(p)\log(2p)}{\log r_1\log(2(1-p))-\log r_2\log(2p)},
\]
where
\[
c_1(p):=\frac{\log(2(1-p))}{\log(1-p)-\log p}, \qquad c_2(p):=\frac{\log(2p)}{\log(p)-\log(1-p)}=1-c_1(p).
\]
There are thresholds $p_*$ and $p^*$ satisfying $0<p_*<1/2<p^*<1$ such that, with probability one,
\[
\dim_H F=\dim_B F=\begin{cases}
s_0 & \mbox{if $p_*\leq p\leq 1/2$},\\
\tilde{s}(p) & \mbox{if $1/2\leq p\leq p^*$},\\
\hat{s}(p) & \mbox{if $p\leq p_*$ or $p\geq p^*$}.
\end{cases}
\]
Specifically, $p_*$ is the unique solution in $(0,1/2)$ of the equation
\[
\frac{\log(2p)}{\log(2(1-p))}=-\left(\frac{r_2}{r_1}\right)^{s_0},
\]
and $p^*$ is the first coordinate of the unique pair $(p,s)$ satisfying $1/2<p<1$ and the system of equations
\begin{gather*}
pr_1^s+(1-p)r_2^s=\frac12,\\
pr_1^s\log p+(1-p)r_2^s\log(1-p)=-\frac12\log 2.
\end{gather*}
Thus, $\dim_H F$, as a function of $p$, is strictly increasing on $[0,p_*]$, constant on $[p_*,1/2]$, and strictly decreasing on $[1/2,1]$.
\end{theorem}

The graph of $\dim_H F$ as a function of $p$ (in the setting of Theorem \ref{cor:two-maps-nonhomogeneous}) is shown in Figure \ref{fig:non-homogeneous} for $r_1=0.2, r_2=0.7$. The phase transitions at $p_*$ and $1/2$ are clearly visible, but the one at $p^*$ is more difficult to see. It can be verified using calculus that $\hat{s}'(p_*)=0$ and $\tilde{s}'(1/2)<0$, so the phase transition at $p_*$ is differentiable but the one at $1/2$ is not. In fact, with additional calculations it can be shown that $\hat{s}''(p_*)<0$, so that the phase transition at $p_*$ is not twice differentiable. Furthermore, it can be shown (see Lemma \ref{lem:convex}) that the function $\tilde{s}(p)$ is strictly convex on $1/2\leq p<1$. It also appears from the graph that $\hat{s}(p)$ is strictly concave on $[0,1]$. Assuming this, it follows that the phase transition at $p^*$ also is once, but not twice differentiable.

\begin{figure}[h] 
\begin{center}
\begin{tikzpicture}[xscale=10,yscale=8]
\draw [->] (0,0) node[anchor=north] {$0$}  -- (1.05,0) node[anchor=west] {$p$};
\draw [->] (0,0) node[anchor=east] {$0$} -- (0,0.9) node[anchor=south] {$\dim_H F$};
\draw(1,0) node[anchor=north]{$1$};
\draw(0.5,0) node[anchor=north]{$1/2$};
\draw[dashed,blue](0.5,0)--(0.5,0.8398);
\draw(0.09,-0.008) node[anchor=north]{$p_*$};
\draw[dashed,blue](0.09,0)--(0.09,0.8398);
\draw(0.79,0) node[anchor=north]{$p^*$};
\draw[dashed,blue](0.79,0)--(0.79,0.5474);
\draw[thick] (0,0)--(0.0001,0.5921)--(0.001,0.6751)--(0.005,0.7437)--(0.01,0.7743)--(0.02,0.8032)--(0.03,0.8181)--(0.04,0.827)--(0.05,0.8326)--(0.06,0.8362)--(0.07,0.8383)--(0.08,0.8395)--(0.09,0.8398)--(0.5,0.8398)--(0.52,0.8121)--(0.54,0.7859)--(0.56,0.7609)--(0.58,0.7371)--(0.6,0.7145)--(0.62,0.693)--(0.64,0.6726)--(0.66,0.6532)--(0.68,0.6347)--(0.7,0.6171)--(0.72,0.6004)--(0.74,0.5844)--(0.77,0.5618)--(0.79,0.5476)--(0.8,0.5406)--(0.82,0.5259)--(0.84,0.5104)--(0.86,0.4937)--(0.88,0.4756)--(0.9,0.4558)--(0.92,0.4334)--(0.94,0.4074)--(0.95,0.3924)--(0.96,0.3754)--(0.97,0.3554)--(0.98,0.3306)--(0.99,0.2954)--(0.995,0.2671)--(0.999,0.2194)--(1,0);
\draw[dashed,thick] (0.09,0.8398)--(0.14,0.8343)--(0.19,0.8224)--(0.24,0.8072)--(0.29,0.7901)--(0.34,0.7716)--(0.39,0.752)--(0.44,0.7313)--(0.49,0.7096)--(0.54,0.6868)--(0.59,0.6627)--(0.64,0.6372)--(0.69,0.6099)--(0.74,0.5803)--(0.79,0.5476);
\draw(0.65,0.75) node {$\tilde{s}(p)$};
\draw(0.94,0.5) node {$\hat{s}(p)$};
\draw(0.047,0.75) node {$\hat{s}(p)$};
\draw(0.29,0.88) node {$s_0$};
\end{tikzpicture}
\end{center}
\caption{The graph of $\dim_H F$ as a function of $p$ in Theorem \ref{cor:two-maps-nonhomogeneous}, for $r_1=0.2$ and $r_2=0.7$. Here $s_0\approx 0.8398$, $p_*\approx 0.0900$ and $p^*\approx 0.7903$. The dashed curve between $p_*$ and $p^*$ is the continuation of the graph of $\hat{s}(p)$ where it lies strictly below $\dim_H F$.}
\label{fig:non-homogeneous}
\end{figure}


\begin{remark}
{\rm
That the Hausdorff and box dimensions of $F$ coincide follows since $F$ is ``statistically self-similar" in the sense of Graf \cite{Graf} (see \cite[Section 2.2]{Allaart-Jones}), but it is also a consequence of our proof, since we will estimate the Hausdorff dimension from below, and the box dimension from above.
}
\end{remark}

\subsection{Connections with the literature}

There is a vast literature on random fractals, and we limit ourselves here to the works most closely related to our model. 
Statistically self-similar sets were introduced independently in the 1980s by Falconer \cite{Falconer}, Graf \cite{Graf}, and Mauldin and Williams \cite{Mauldin-Williams}. These models belong to a class of {\em infinite-variable} random fractals, where at each stage of the construction one chooses for each cylinder set at that stage an IFS at random from a given collection according to a given probability distribution, independently of other cylinders at that stage and of previous stages. This class also includes the model of {\em fractal percolation}, where one divides for instance a square into $M^2$ subsquares (for some fixed $M\geq 2$) and retains each subsquare with probability $p$, or deletes it with probability $1-p$, independently of other subsquares, and keeps repeating this procedure ad infinitum. 

By contrast, in {\em one-variable} random fractal models one chooses at each stage of the construction a single IFS at random and applies it to all cylinders at that stage. Hence, one-variable models are closer in nature to deterministic self-similar sets. To fill the void between the one-variable and the infinite-variable case, Barnsley et al.~\cite{Barnsley} introduced the {\em $V$-variable} random fractals, which were further generalized by J\"arvenp\"a\"a et al.~\cite{Jarvenpaa}. Generalizing in a different direction, Troscheit \cite{Troscheit1} computes various fractal dimensions of both one-variable and infinite-variable random graph-directed self-similar sets.

In several of the above models, there is a positive probability that the limit set is empty, and one has to compute the dimension conditioned on non-extinction. The model we consider here precludes this possibility. 

All of the above-mentioned papers assume some type of separation condition when computing the Hausdorff dimension of the limit set. Much less is known in the case of random self-similar sets with overlaps, although Koivusalo \cite{Koivusalo} manages to compute the Hausdorff dimension in a certain ``uniform" model of random self-similar fractals without assuming any separation condition. In the present article, although we assume the open set condition for our initial IFS, the random subset $F$ can be viewed as a statistically self-similar set with ``extreme overlaps" - see \cite{Allaart-Jones} for an explanation. (In essence, we choose randomly from a collection of IFSs which may have one or more maps repeated.) As such, we believe that this article is a useful step towards a better understanding of random fractals with overlaps.

Aside from \cite{Allaart-Jones}, the only other paper we have found that considers a setup similar to ours is the work by Benjamini et al.~\cite{BGS}. Motivated by the study of Bernoulli convolutions, they consider an overlapping IFS of two maps and a randomly labeled binary tree. Rather than the random attractor $F$, their interest is mainly in the resulting random self-similar measure $\mu$ supported on $F$; that is, the weak limit as $n\to\infty$ of the distribution of the particles at stage $n$ of the branching random walk.

\section{Definitions and statements of main results} \label{sec:result}

We repeat the setup from the Introduction in a more formal manner. Consider an iterated function system (IFS) $\mathcal{F}=\{f_1,\dots,f_N\}$ of similar contractions on $\RR^d$ ($d\in\NN$), where $N\geq 2$. We assume that there is a non-empty bounded open set $U\subset\RR^d$ such that
\begin{enumerate}[(O1)]
\item $f_i(U)\subset U$ for $i=1,\dots,N$;
\item $f_i(U)\cap f_j(U)=\emptyset$ for all $i\neq j$; and
\item $U$ is the interior of its closure.
\end{enumerate}
Conditions (O1) and (O2) are the usual open set condition. The extra condition (O3), which is satisfied by many well-known examples of self-similar sets, will be needed to prove the lower bound for the Hausdorff dimension of our random subset. At the moment we do not know how to eliminate this condition.

Let $\Lambda$ be the attractor of the IFS; that is, $\Lambda$ is the unique non-empty compact set such that
\[
\Lambda=\bigcup_{i=1}^N f_i(\Lambda).
\]
Then the Hausdorff dimension of $\Lambda$ is the unique real number $s_0$ satisfying
\begin{equation} \label{eq:Moran}
\sum_{i=1}^N r_i^{s_0}=1,
\end{equation}
where $r_i$ is the contraction ratio of $f_i$. 

We construct a random subset of $\Lambda$ in the following way. Let $M\geq 2$, and consider a full infinite $M$-ary tree $\mathcal{T}_M$; that is, a tree in which one vertex, the {\em root}, has degree $M$ and all other vertices have degree $M+1$. Let $E$ be the set of edges of $\mathcal{T}_M$, and put $\Omega:=\{1,2,\dots,N\}^E$. We equip $\{1,2,\dots,N\}$ with the discrete topology and $\Omega$ with the corresponding product topology. This induces a Borel $\sigma$-algebra $\mathcal{B}$ on $\Omega$. We now let $\PP$ be the Bernoulli measure on $(\Omega,\mathcal{B})$ associated with the probability vector $\pp$. That is, under $\PP$, the random variables $\omega(e), e\in E$ are independent with $\PP(\omega(e)=i)=p_i$, $i=1,2,\dots,N$. We will let $\EE$ denote the corresponding expectation operator.

Let $\mathcal{P}$ denote the set of all infinite paths in $\mathcal{T}_M$ that start at the root. For a path $P\in\CP$ and $k\in\NN$, we denote by $P(k)$ the $k$th edge of $P$.


For $\omega\in\Omega$, define the set
\[
F(\omega):=\Bigg\lbrace x=\lim_{n\to\infty} f_{\omega(P(1))} \circ \dots \circ f_{\omega(P(n))}(0) \Big| P \in \mathcal{P}\Bigg\rbrace.
\]
Clearly, $F(\omega)\subset \Lambda$, giving the trivial upper bound $\dim_H F\leq s_0$. Our approach in this article is to express the almost-sure dimension of $F$ as the solution of a constrained maximization problem, which we then solve using elementary calculus.


In this paper, ``$\log$" denotes the natural logarithm, and we make the usual convention that $0\log 0=0$. We define the following functions from $[0,\infty)^N$ to $\RR$:
\begin{gather}
\psi(x_1,\dots,x_N):=\sum_{i=1}^N x_i\left(\log\sum_{j=1}^N x_j-\log x_i\right), \label{eq:psi}\\
\phi(x_1,\dots,x_N):=\sum_{i=1}^N x_i\log(Mp_i), \label{eq:phi}\\
h(x_1,\dots,x_N):=-\sum_{i=1}^N x_i\log r_i. \label{eq:h}
\end{gather}
Write $\xx=(x_1,\dots,x_N)$ for short.

\begin{theorem} \label{thm:variational}
We have, $\PP$-almost surely,
\begin{equation} \label{eq:variational-principle}
\dim_H F=\dim_B F=\max\big\{\psi(\xx)+\min\{0,\phi(\xx)\}: \xx\in[0,\infty)^N,\ h(\xx)=1\big\}.
\end{equation}
\end{theorem}

Proving this result will be our main task in the remainder of the paper.  In Section \ref{sec:sketch}, we sketch its proof and explain where the different terms in the optimization problem come from. The constrained maximum cannot exceed the dimension of $\Lambda$, because
\[
\dim_H \Lambda=\dim_B \Lambda=\max\big\{\psi(\xx)
: \xx\in[0,\infty)^N,\ h(\xx)=1\big\}.
\]
This last equation may be verified readily using Lagrange multipliers and the Moran equation \eqref{eq:Moran}. An immediate consequence is that, if $Mp_i\geq 1$ for each $i$, then $\dim_H F=\dim_H\Lambda$ almost surely. We will see in the next theorem that this is the case already under weaker conditions.

Theorem \ref{thm:variational} reduces the problem of calculating the dimension of $F$ to solving an explicit optimization problem, which we do in Section \ref{sec:main-proof} to get our main result, the explicit Theorem \ref{thm:main}.  Note that the constrained maximum is well defined since the domain $[0,\infty]^N\cap \{\xx: h(\xx)=1\}$ is compact.  Because of the presence of a minimum in the objective function, the solution of the optimization problem breaks into three cases: We have to consider the maximum of $\psi$, the maximum of $\psi+\phi$, and the constrained maximum of $\psi$ subject to $\phi(\xx)=0$, all under the further constraint $h(\xx)=1$. (This explains the three cases in Theorem \ref{cor:two-maps-nonhomogeneous}.) Doing so leads to the following general result, in which $\tilde{s}$ is the solution of
\begin{equation} \label{eq:weighted-Moran}
\sum_{i=1}^N p_i r_i^{\tilde{s}}=\frac{1}{M}.
\end{equation}

\begin{theorem} \label{thm:main}
The Hausdorff dimension of the random subset $F$ is determined as follows:
\begin{enumerate}[(i)]
\item If
\begin{equation} \label{eq:M-large-condition}
\sum_{i=1}^N r_i^{s_0}\log(Mp_i)\geq 0,
\end{equation}
then $\dim_H F=s_0$, $\PP$-almost surely.
\item If
\begin{equation} \label{eq:M-small-condition}
\sum_{i=1}^N p_i r_i^{\tilde{s}}\log(Mp_i)\leq 0,
\end{equation}
then $\dim_H F=\tilde{s}\leq s_0$, $\PP$-almost surely, with equality if and only if $M=N$ and $p_i=1/N$ for $i=1,\dots,N$.
\item If
\begin{equation} \label{eq:main-case}
\sum_{i=1}^N r_i^{s_0} \log(Mp_i)<0<\sum_{i=1}^N p_i r_i^{\tilde{s}}\log(Mp_i),
\end{equation}
then there is a unique pair $(\hat{s},\hat{t})$ of real numbers, with $\hat{s}>0$, satisfying the equations
\begin{gather}
\sum_{i=1}^N r_i^{\hat{s}}(Mp_i)^{\hat{t}}=1, \label{eq:s-t-equation-a}\\
\sum_{i=1}^N r_i^{\hat{s}}(Mp_i)^{\hat{t}}\log(Mp_i)=0, \label{eq:s-t-equation-b}
\end{gather}
and $\dim_H F=\hat{s}<s_0$, $\PP$-almost surely.
\end{enumerate}
\end{theorem}

\begin{remark}
{\rm
The conditions \eqref{eq:M-large-condition} and \eqref{eq:M-small-condition} are not mutually exclusive, but if both hold then we must have that $M=N$ and $p_i=1/N$ for $i=1,\dots,N$, in which case $\tilde{s}=s_0$. 
}
\end{remark}

The theorem simplifies considerably if the IFS is homogeneous, i.e. if all the maps have the same contraction ratio. The stated dimension in the next corollary corresponds to the upper bound of Allaart and Jones \cite{Allaart-Jones}.

\begin{corollary}[The homogeneous case] \label{cor:homogeneous}
Assume $r_i=r\in(0,1)$ for $i=1,\dots,N$. Define the two quantities
\[
L(N,\pp):=\prod_{i=1}^N p_i^{-p_i}, \qquad U(N,\pp):=\prod_{i=1}^N p_i^{-1/N},
\]
and observe that $L(N,\pp)\leq N\leq U(N,\pp)$. We have, $\PP$-almost surely,
\[
\dim_H F=\dim_B F=-\frac{\log\rho}{\log r},
\]
where $\rho$ is determined as follows:
\begin{enumerate}[(i)]
\item If $M\leq L(N,\pp)$, then $\rho=M$.
\item If $L(N,\pp)\leq M\leq U(N,\pp)$, then 
\[
\rho=\sum_{i=1}^N (Mp_i)^\lambda,
\]
where $\lambda\in[0,1]$ is the unique solution of
\begin{equation} \label{eq:mu-equation}
\sum_{i=1}^N p_i^\lambda\log(Mp_i)=0.
\end{equation}
\item If $M\geq U(N,\pp)$, then $\rho=N$.
\end{enumerate}
\end{corollary}

\begin{remark}
{\rm
The hypothesis of case (ii) is always satisfied when $M=N$, since
\[
\sum_{i=1}^N p_i\log p_i\geq -\log N=\log\bigg(\frac{1}{N}\sum_{i=1}^N p_i\bigg)\geq \frac{1}{N}\sum_{i=1}^N\log p_i,
\]
by concavity of the logarithm. But the range of values of $M$ satisfying $L(N,\pp)\leq M\leq U(N,\pp)$ can be much larger. For instance, take $N=3$ and $\pp=(0.05,0.2,0.75)$. Then $\prod_{i=1}^N p_i^{-p_i}\approx 1.9886$ and $\big(\prod_{i=1}^N p_i\big)^{-1/N}\approx 5.1087$, so $M$ could be $2,3,4$ or $5$. 
}
\end{remark}

\begin{remark}
{\rm
Consider the homogeneous case: $r_i=r$ for each $i$, and assume $M=N$ and $p_i=1/N$ for $i=1,\dots,N$. Then Corollary \ref{cor:homogeneous} implies $\dim_H F=s_0=-\log N/\log r$. It was shown in \cite[Proposition 5.1]{Allaart-Jones} that $\mathcal{H}^{s_0}(F)=0$ almost surely in this case. We suspect that $\mathcal{H}^{s_0}(F)>0$ when $M>N$, but have not been able to prove this. (On the other hand, something even stronger can be shown when $M\geq 2N$, namely with $\PP$-probability 1, $F$ contains a similar copy of the full attractor $\Lambda$ of the IFS $\{f_1,\dots,f_N\}$, and $F$ is even equal to all of $\Lambda$ with positive probability -- see \cite[Remark 5.6]{Allaart-Jones}.)
}
\end{remark}

The rest of this article is organized as follows: In Section \ref{sec:sketch} we sketch the proof of Theorem \ref{thm:variational} to give the reader an intuitive understanding.  In Section \ref{sec:stopping-set} we show that the dimension of $F$ is given by the exponential growth rate of the expected cardinality of a certain random stopping set $\tilde{\Gamma}_n$. The estimate from below uses the main theorem of Mauldin and Williams \cite{Mauldin-Williams} on the Hausdorff dimension in random recursive constructions. 
We derive an upper estimate for the cardinality of $\tilde{\Gamma}_n$ in Section \ref{sec:upper-bound}, and a lower estimate in Section \ref{sec:lower-bound}. In Section \ref{sec:main-proof} we solve the constrained optimization problem of Theorem \ref{thm:variational} to deduce Theorem \ref{thm:main}. Finally, in Section \ref{sec:corollaries}, we prove the special cases given in Theorem \ref{cor:two-maps-homogeneous}, Theorem \ref{cor:two-maps-nonhomogeneous} and Corollary \ref{cor:homogeneous}.

\section{Sketch of proof of Theorem \ref{thm:variational}} \label{sec:sketch}

Here we sketch the main ideas of the proof and explain where the maps $\psi$, $\phi$ and $h$ come from. This section may be skipped without loss of continuity.

To start, we sketch why 
\begin{equation} \label{eq:var-set}
\dim_B \Lambda=\max\big\{\psi(\xx): \xx\in[0,\infty)^N,\ h(\xx)=1\big\},
\end{equation}
where $\psi(\xx)$ and $h(\xx)$ are defined by \eqref{eq:psi} and \eqref{eq:h}, respectively. Although this identity can be deduced from the Moran equation using the method of Lagrange multipliers, we give a more direct explanation here because it will help us introduce the main concepts of the proof of Theorem \ref{thm:variational}.
For ease of exposition,  we assume that $\Lambda \subset \RR$ is a self-similar set satisfying the open set condition with the open set $U$ being an  interval. Put $I:=\overline{U}$.

For each $n\in\NN$, we cover $\Lambda$ by intervals of length $\approx e^{-n}$. More precisely, define the stopping set
\begin{equation} 
\Gamma_n:=\{\ii=i_1\dots i_l\in\{1,\dots,N\}^*: r_{i_1}\dots r_{i_l}\leq e^{-n}<r_{i_1}\dots r_{i_l-1}\}.
\end{equation}
The intervals $I_{\ii}:=f_{i_1} \circ\dots \circ f_{i_l}(I)$, for $\ii=i_1\dots i_l\in\Gamma_n$ thus form an $e^{-n}$-cover of $\Lambda$, and these intervals do not overlap due to the open set condition. We need to estimate the cardinality of this cover. Suppose that $\ii\in\Gamma_n$ and that for each $j$, the letter $j$ occurs $k_j$ times in the word $\ii$; that is,
\begin{equation} \label{eq:frequencies}
\#\{t\leq |\ii|:i_t=j\}=k_j,\qquad j=1,\dots,N. 
\end{equation}
Then $\ii\in\Gamma_n$ implies
\begin{equation} \label{eq:kisize}
r_{\min}e^{-n} \leq \prod_i r_i^{k_i} \leq e^{-n},
\end{equation}
where $r_{\min}:=\min_{1\leq i\leq N} r_i$. Let $\Gamma_n(k_1,\dots,k_N)$ denote the set of all words $\ii\in\Gamma_n$ satisfying \eqref{eq:frequencies}. Note these words all have the same length $l=k_1+\dots+k_N$. By Stirling's approximation, we have
\[
\#\Gamma_n(k_1, \dots, k_N)=\binom{k_1+\dots+k_N}{k_1,\dots,k_N} \approx \prod_{i=1}^N \alpha_i^{-k_i},
\]
where $\alpha_i:=\frac{k_i}{k_1+\dots+k_N}$ and the $\approx$ is understood to be up to polynomial factors (see Lemma \ref{lem:multinomial-estimate} for a more precise version). Setting $x_i:=k_i/n$ and $\xx=(x_1,\dots,x_N)$, we thus obtain
\[
\#\Gamma_n(k_1, \dots, k_N)\approx \prod_{i=1}^N \alpha_i^{-k_i}=e^{n\psi(\xx)}.
\]
Furthermore, there are at most polynomially many distinct tuples $(k_1, \dots, k_N)$ satisfying (\ref{eq:kisize}). 
Hence, again ignoring polynomial factors, 
\begin{align*}
\#\Gamma_n &\approx \max\{\#\Gamma_n(k_1, \dots, k_N): (k_1,\dots,k_N)\ \mbox{satisfies \eqref{eq:kisize}}\}\\
& \approx \exp\big(n\max\{\psi(\xx): h(\xx)=1\}\big),
\end{align*}
for large $n$, since (\ref{eq:kisize}) translates into $h(\xx)=1+O(1/n)$. From these asymptotics, \eqref{eq:var-set} follows.

How does this change when one instead considers the random subset $F$ of $\Lambda$? To sketch the proof of the upper bound and explain how the function $\phi$ comes in, we fix again a scale $n$ and observe that for a word $\ii\in\Gamma_n$, the interval $I_\ii$ is only needed in a covering of $F$ if $I_\ii\cap F\neq\emptyset$. Hence we consider the (random) subcollection
\begin{align*}
\tilde{\Gamma}_n:&=\{\ii\in\Gamma_n: I_\ii\cap F\neq\emptyset\}\\
&=\{\ii\in\Gamma_n:\ \mbox{at least one path in $\mathcal{P}_{|\ii|}$ is labeled $\ii$}\}.
\end{align*}
In parallel with the above, we also let $\tilde{\Gamma}_n(k_1,\dots,k_N)$ denote the set of all words $\ii\in\tilde{\Gamma}_n$ satisfying \eqref{eq:frequencies}.
Put
\[
a_{\mathbf{i}}:=\PP(I_{\mathbf{i}} \cap F \neq \emptyset).
\]
Then
\begin{align*}
\EE(\#\tilde{\Gamma}_n)&=\sum_{k_1,\dots,k_N} \EE\big(\#\tilde{\Gamma}_n(k_1,\dots,k_N)\big)\\
&\lesssim \max_{k_1,\dots,k_N} \EE\big(\#\tilde{\Gamma}_n(k_1,\dots,k_N)\big)\\
&=\max_{k_1,\dots,k_N} \sum_{\ii\in \Gamma_n(k_1, \dots, k_N)} a_\ii,
\end{align*}
where $\lesssim$ is again to be interpreted up to a polynomial factor, and the first maximum is over tuples $(k_1,\dots,k_N)$ satisfying \eqref{eq:kisize}.
Now for $\ii\in \Gamma_n(k_1, \dots, k_N)$ with length $l=k_1+\dots+k_N$, we estimate
\begin{align}
a_\ii&=\PP\left(\bigcup_{P\in \CP_l}\big\{\mbox{path $P$ gets label sequence $\ii$}\big\}\right) \label{eq:a_i-pointwise}\\
&\leq \min\left\{1,\sum_{P\in \CP_l} \PP\big(\mbox{path $P$ gets label sequence $\ii$}\big)\right) \notag\\
&\leq \min\left\{1,M^l\prod_{j=1}^N p_j^{k_j}\right\}=\min\left\{1,\prod_{j=1}^N (Mp_j)^{k_j}\right\}.\notag
\end{align}
Combining the last two estimates, we obtain
\begin{align}
\EE(\#\tilde{\Gamma}_n) &\lesssim \max_{k_1,\dots,k_N} \binom{k_1+\dots+k_N}{k_1,\dots,k_N} \min\left\{1,\prod_{j=1}^N (Mp_j)^{k_j}\right\} \label{eq:discrete-maximum}\\
&\approx \exp\big(n\max\{\psi(\xx)+\min\{0, \phi(\xx)\}: h(\xx)=1\}\big), \notag
\end{align}
which explains the presence of $\phi(\xx)$ in Theorem \ref{thm:variational}. The proof of the upper bound is completed by a standard application of Markov's inequality and the Borel-Cantelli lemma; see Section \ref{sec:upper-bound} for the details.


We next sketch the proof of the lower bound.  Adapting a technique of Liu and Wu \cite{Liu-Wu}, we first construct a sequence $\{F_n\}$ of (random) subsets of $F$ such that for each $n$, $F_n$ is the limit set of a Mauldin-Williams random recursive construction (see \cite{Mauldin-Williams}) and satisfies
\begin{equation} 
\dim_H F_n\sim\frac{\log \EE(\#\tilde{\Gamma}_n)}{n}.
\end{equation}
The details of this construction are given in Section \ref{sec:stopping-set}.
Letting $n\to\infty$ then yields
\[
\dim_H F\geq \liminf_{n\to\infty} \frac{\log \EE(\#\tilde{\Gamma}_n)}{n}.
\]
Since $\EE(\#\tilde{\Gamma}_n)=\sum_{\ii\in\Gamma_n}a_\ii$, the task then turns to estimating the probabilities $a_{\mathbf{i}}$ from below.  This would be easy if the paths in the tree were completely independent, as $a_{\mathbf{i}}$ would then be constant on $\Gamma_n(k_1, \dots, k_N)$,  but the cross-dependence of the different paths due to shared edges can make some of the $a_{\mathbf{i}}$ exponentially small.  For example,  if $p_1$ is such that $Mp_1<1$ and $\ii$ is such that $i_1=\dots=i_{k_1}=1$, then by reasoning analogous to what we used for the upper bound,
\[
a_\ii\leq (Mp_1)^{k_1},
\]
and since each $k_i$ is of order $n$ for the tuple $(k_1,\dots,k_N)$ maximizing the expression in \eqref{eq:discrete-maximum}, this shows that $\min_{\ii\in\Gamma_n}a_\ii\to 0$ exponentially fast. Thus, in order to estimate $\sum_{\ii\in\Gamma_n}a_\ii$ from below, we need to exclude highly ordered words as in the above example. This motivates the definition of the special subset $D_n\subset\Gamma_n$ in the proof (see Section \ref{sec:lower-bound}),  which roughly speaking consists of the sufficiently ``random-looking" sequences $\mathbf{i}$.  Via standard estimates we can show that the number of these ``random-looking" sequences is not much smaller than the cardinality of $\Gamma_n(k_1,\dots,k_N)$, and for elements in $D_n$,  we are able to show a sufficient lower bound on $a_{\mathbf{i}}$ via an iteration.

\section{The random stopping set} \label{sec:stopping-set}

For each finite word $\ii\in\{1,\dots,N\}^*$, let $J_\ii:=f_\ii(\bar{U})$, where we write $f_\ii=f_{i_1}\circ\dots\circ f_{i_n}$ for $\ii=(i_1,\dots,i_n)$. We also let $\mathcal{P}_l$ denote the set of all paths of length $l$ in the $M$-ary tree starting from the root, for $l\in\NN$. To simplify the presentation somewhat, we assume without loss of generality that $|U|=1$. Denote
\[
r_{\min}:=\min_{1\leq i\leq N}r_i \qquad\mbox{and}\qquad r_{\max}:=\max_{1\leq i\leq N}r_i.
\]
For each $n\in\NN$, we define the stopping set
\begin{equation} \label{eq:stopping-set}
\Gamma_n:=\{\ii=i_1\dots i_l\in\{1,\dots,N\}^*: r_{i_1}\dots r_{i_l}\leq e^{-n}<r_{i_1}\dots r_{i_l-1}\}.
\end{equation}
Observe that the collection $\{J_\ii: \ii\in\Gamma_n\}$ is an $e^{-n}$-cover of $\Lambda$, hence of $F$. However, not all of these sets are needed to cover $F$, so we define also the smaller (random) set
\begin{equation} \label{eq:stopping-subset}
\tilde{\Gamma}_n:=\{\ii\in\Gamma_n:\ \mbox{at least one path in $\mathcal{P}_{|\ii|}$ is labeled $\ii$}\}.
\end{equation}
Then $\{J_\ii: \ii\in\tilde{\Gamma}_n\}$ is a more efficient $e^{-n}$-cover of $F$. Since the sets in this cover are roughly the same size, we need to estimate their number. In fact, as the next proposition shows, it is sufficient to determine the exponential growth rate of the {\em expected} cardinality of $\tilde{\Gamma}_n$. Let $\EE$ denote the expectation operator associated with the probability measure $\PP$.

\begin{proposition} \label{prop:growth-rate}
We have, $\PP$-almost surely,
\[
\liminf_{n\to\infty} \frac{\log \EE(\#\tilde{\Gamma}_n)}{n}\leq \dim_H F\leq \overline{\dim}_B F\leq \limsup_{n\to\infty} \frac{\log \EE(\#\tilde{\Gamma}_n)}{n}.
\]
\end{proposition}

\begin{proof}
We first prove the upper bound. Let
\[
\lambda:=\limsup_{n\to\infty} \frac{\log \EE(\#\tilde{\Gamma}_n)}{n}.
\]
Fix $\eps>0$. Put $Z_n:=\#\tilde{\Gamma}_n$. Then, for all large enough $n$,
\[
\PP\big(Z_n>e^{(\lambda+2\eps)n}\big)\leq \frac{\EE(Z_n)}{e^{(\lambda+2\eps)n}}\leq \frac{e^{(\lambda+\eps)n}}{e^{(\lambda+2\eps)n}}=e^{-\eps n}.
\]
As the series $\sum_{n=1}^\infty \PP\big(Z_n>e^{(\lambda+2\eps)n}\big)$ converges,  by the Borel-Cantelli lemma, $Z_n\leq e^{(\lambda+2\eps)n}$ for all sufficiently large $n$ with probability one. Therefore, since $f_{\ii}(U)$ has diameter approximately $e^{-n}$ for any $\ii\in\Gamma_n$,
\[
\overline{\dim}_B F\leq \lambda+2\eps \qquad\mbox{$\PP$-almost surely},
\]
and letting $\eps\to 0$ along a discrete sequence gives
$\overline{\dim}_B F\leq \lambda$, $\PP$-almost surely.

For the lower bound, we borrow a technique of Liu and Wu \cite{Liu-Wu}. The idea is to construct a sequence $\{F_n\}$ of (random) subsets of $F$ such that for each $n$, $F_n$ is the limit set of a Mauldin-Williams random recursive construction and satisfies
\begin{equation} \label{eq:subset-dimension}
\dim_H F_n\sim\frac{\log \EE(Z_n)}{n}.
\end{equation}
Letting $n\to\infty$ then yields the desired lower bound.

For the remainder of this proof, let $\CP^*$ be the set of all finite paths in $\mathcal{T}_M$ that start at the root.
Fix $n\in\NN$. For $\omega\in\Omega$, let $\CL_1=\CL_1(\omega)=\tilde{\Gamma}_n$. Let $\CB_1=\CB_1(\omega)$ be the set of finite paths $P=(P(1),\dots,P(l))\in\CP^*$ which receive a label sequence from $\CL_1(\omega)$ and such that no path of the same length as $P$ which is lexicographically smaller than (i.e. to the left of) $P$ receives the same label sequence as $P$. Thus, $\#\CB_1=\#\CL_1$. Observe also that no path in $\CB_1$ extends any other, because no word in $\CL_1$ extends any other. (In terms of the branching random walk, if there are two or more particles occupying the same interval of length approximately $e^{-n}$, we remove all but one of them before continuing the process.)

Next, we proceed inductively. For $k=1,2,\dots$, assume $\CB_k$ and $\CL_k$ have been constructed. 	Let $\CB_{k+1}$ be the set of paths $P=P'P''$ in $\CP^*$ such that $P'\in\CB_k$ and $P''$ receives a label sequence in $\Gamma_n$, and furthermore, no path of the same length as $P$ which is lexicographically smaller than $P$ receives the same label sequence as $P$. (We continue to remove ``duplicate" particles in the branching random walk.)

Let $\CL_{k+1}$ be the set of all label sequences of paths in $\CB_{k+1}$. Thus (by induction), $\#\CB_{k+1}=\#\CL_{k+1}$. For each $\ii\in\CL_k$, let $Y_\ii$ be the number of words in $\CL_{k+1}$ that begin with $\ii$.

For each $k$, by construction, the sets $J_\ii:=f_\ii(\bar{U}): \ii\in\CL_k$, are non-overlapping since the IFS $\{f_1,\dots,f_N\}$ satisfies the open set condition. Furthermore, for $\ii\in\CL_k$,
\[
\#\{\ii'\in\CL_{k+1}: J_{\ii'}\subset J_\ii\}=Y_\ii.
\]
Note that the random variables $Y_\ii: \ii\in\bigcup_{k=1}^\infty \CL_k$ are i.i.d. with the same distribution as $Z_n$. Hence, the set
\[
F_n:=\bigcap_{k=1}^\infty \bigcup_{\ii\in\CL_k}J_\ii
\]
is the limit set of a Mauldin-Williams random recursive construction (see \cite{Mauldin-Williams}). Clearly, $F_n\subset F$. Since $\PP(Z_n\geq 1)=1$, $F_n$ is nonempty with probability one. By the main theorem of \cite{Mauldin-Williams}, the Hausdorff dimension of $F_n$ is the unique number $s_n$ satisfying
\[
\EE\left(\sum_{\ii\in\tilde{\Gamma}_n}(r_{i_1}\dots r_{i_{|\ii|}})^{s_n}\right)=1.
\]
(By condition (O3), the set $K:=\bar{U}$ satisfies the assumption of Mauldin and Williams that $K$ is the closure of its interior.)
By the definition of $\Gamma_n$, this implies
\[
\frac{\log \EE(Z_n)}{n-\log r_{\min}}\leq\dim_H F_n=s_n\leq \frac{\log \EE(Z_n)}{n},
\]
from which \eqref{eq:subset-dimension} follows.
\end{proof}

In view of Proposition \ref{prop:growth-rate},  to complete the proof of Theorem \ref{thm:variational}, it remains to show that
\begin{equation} \label{eq:liminf-limsup-sandwich}
\limsup_{n\to\infty} \frac{\log \EE(\#\tilde{\Gamma}_n)}{n}\leq s^*\leq \liminf_{n\to\infty} \frac{\log \EE(\#\tilde{\Gamma}_n)}{n}
\end{equation}
(and hence, the limit of $n^{-1}\log \EE(\#\tilde{\Gamma}_n)$ actually exists),
where $s^*$ is the right-hand side of \eqref{eq:variational-principle}, i.e. 
\[
s^*:=\max\big\{\psi(\xx)+\min\{0,\phi(\xx)\}: \xx\in[0,\infty)^N,\ h(\xx)=1\big\},
\]
where $\psi$, $\phi$ and $h$ were defined in \eqref{eq:psi}, \eqref{eq:phi} and \eqref{eq:h}, respectively.
Our proofs of these inequalities are mostly combinatorial and necessarily somewhat technical. We prove the first inequality in Section \ref{sec:upper-bound}, and the second inequality in Section \ref{sec:lower-bound}.

\section{Proof of the upper bound} \label{sec:upper-bound}

In this section we prove the first inequality in \eqref{eq:liminf-limsup-sandwich}:

\begin{proposition} \label{prop:key-upper-estimate}
We have
\[
\limsup_{n\to\infty}\frac{\log \EE(\#\tilde{\Gamma}_n)}{n}\leq s^*.
\]
\end{proposition}

We use the following estimate for multinomial coefficients, which is an easy consequence of Stirling's approximation.

\begin{lemma}[{\cite[Lemma 7.4]{Allaart}}] \label{lem:multinomial-estimate}
Let $m\in\NN$, and let $k_1,\dots,k_N$ be nonnegative integers with $\sum_{i=1}^N k_i=m$. Put $\alpha_i:=k_i/m$ for $i=1,\dots,N$. Then
\begin{equation*}
\binom{m}{k_1,\dots,k_N}\leq 2\sqrt{m}\left(\prod_{i=1}^N \alpha_i^{-\alpha_i}\right)^m,
\end{equation*}
where we use the convention $0^0\equiv 1$.
\end{lemma}

\begin{proof}[Proof of Proposition \ref{prop:key-upper-estimate}]
For $\ii\in\{1,\dots,N\}^*$, let
\[
k_j(\ii):=\#\{\nu\leq|\ii|: i_\nu=j\}, \qquad j=1,\dots,N.
\]
Then
\begin{equation} \label{eq:first-mean-estimate}
\EE(\#\tilde{\Gamma}_n)=\sum_{\ii\in\Gamma_n}\PP(\ii\in\tilde{\Gamma}_n)\leq \sum_{\ii\in\Gamma_n}\min\left\{\prod_{j=1}^N(Mp_j)^{k_j(\ii)},1\right\},
\end{equation}
where the inequality was established in \eqref{eq:a_i-pointwise}.
If $\ii\in\Gamma_n$ and $k_j(\ii)=k_j$ for $j=1,\dots,N$, then
\begin{equation} \label{eq:product-sandwich}
\prod_{j=1}^N r_i^{k_i}\leq e^{-n}<r_{\min}^{-1}\prod_{j=1}^N r_j^{k_j}.
\end{equation}
Let $B_n$ denote the set of $N$-tuples $(k_1,\dots,k_N)$ for which these inequalities hold. By the rightmost inequality in \eqref{eq:product-sandwich},
\[
\sum_{j=1}^N k_j\log r_j>\log r_{\min}-n,
\]
and so (since $\log r_j<0$) certainly $k_j\log r_j>\log r_{\min}-n$ for each $j$. Hence
\begin{equation} \label{eq:k-bound}
k_j\leq \frac{\log r_{\min}-n}{\log r_j}\leq \frac{n+|\log r_{\min}|}{|\log r_{\max}|}, \qquad j=1,\dots,N,
\end{equation}
so there is a constant $C>0$ such that $\#B_n\leq Cn^N$ for all $n\in\NN$. Thus, from \eqref{eq:first-mean-estimate}, we obtain
\begin{align*}
\EE(\#\tilde{\Gamma}_n)&\leq \sum_{(k_1,\dots,k_N)\in B_n}\binom{k_1+\dots+k_N}{k_1,\dots,k_N}\min\left\{\prod_{j=1}^N(Mp_j)^{k_j},1\right\}\\
&\leq Cn^N\max_{(k_1,\dots,k_N)\in B_n}\binom{k_1+\dots+k_N}{k_1,\dots,k_N}\min\left\{\prod_{j=1}^N(Mp_j)^{k_j},1\right\}.
\end{align*}

Put $\beta_i:=k_i/n$; then $\beta_i\geq 0$ and \eqref{eq:product-sandwich} is equivalent to
\begin{equation} \label{eq:beta-sandwich}
\sum_{i=1}^N \beta_i \log r_i\leq -1<\sum_{i=1}^N \beta_i \log r_i-\frac{1}{n}\log r_{\min}.
\end{equation}
Let $B_n'$ denote the set of $N$-tuples $\boldsymbol{\beta}=(\beta_1,\dots,\beta_N)\in[0,\infty)^N$ satisfying these inequalities. Let
\[
\alpha_i:=\frac{k_i}{k_1+\dots+k_N}=\frac{\beta_i}{\beta_1+\dots+\beta_N}, \qquad i=1,\dots,N.
\]
By Lemma \ref{lem:multinomial-estimate}, we have
\[
\binom{k_1+\dots+k_N}{k_1,\dots,k_N}\leq 2(k_1+\dots+k_N)^{1/2}\left(\prod_{i=1}^N \alpha_i^{-\alpha_i}\right)^{k_1+\dots+k_N}.
\]
Note that $k_1+\dots+k_N\leq C'n$ for a suitable constant $C'$, by \eqref{eq:k-bound}. On the other hand, we can write $\alpha_i(k_1+\dots+k_N)=\alpha_i(\sum_j\beta_j)n=\beta_i n$. Thus,
\begin{align*}
\left(\prod_{i=1}^N \alpha_i^{-\alpha_i}\right)^{k_1+\dots+k_N}& \cdot \min\left\{\prod_{i=1}^N(Mp_i)^{k_i},1\right\}\\
&=\left[\left(\sum_i\beta_i\right)^{\sum_i\beta_i}\prod_i \beta_i^{-\beta_i}\cdot \min\left\{\prod_{i=1}^N(Mp_i)^{\beta_i},1\right\}\right]^n\\
&=\exp\big(n\{\psi(\boldsymbol{\beta})+\min\{0,\phi(\boldsymbol{\beta})\}\big).
\end{align*}
It follows that, for a suitable constant $C''$,
\begin{equation} \label{eq:Gamma-bound}
\EE(\#\tilde{\Gamma}_n) \leq 2C''n^N\sqrt{n}\sup_{\boldsymbol{\beta}\in B_n'} \exp\left(n\{\psi(\boldsymbol{\beta})+\min\{0,\phi(\boldsymbol{\beta})\}\right).
\end{equation}

Since the extra term $(1/n)\log r_{\min}$ in \eqref{eq:beta-sandwich} goes to zero as $n\to\infty$, we can replace the condition of $\boldsymbol{\beta}$ satisfying \eqref{eq:beta-sandwich} by $h(\boldsymbol{\beta})=-\sum\beta_i\log r_i=1$ without changing the last estimate by much. Specifically, we show:

\bigskip
{\bf Claim:} For each $\eps>0$ there exists an integer $n_0$ such that for each $n\geq n_0$ and each $\boldsymbol{\beta}\in B_n'$, there exists a vector $\boldsymbol{\beta}'\in[0,\infty)^N$ such that $h(\boldsymbol{\beta}')=1$ and
\[
\psi(\boldsymbol{\beta})+\min\{0,\phi(\boldsymbol{\beta})\}\leq \psi(\boldsymbol{\beta}')+\min\{0,\phi(\boldsymbol{\beta}')\}+2\eps.
\]

{\bf Proof of Claim:} Let $\eps>0$ be given. For any $\boldsymbol{\beta}\in B_n$, we have, as in \eqref{eq:k-bound},
\[
0\leq\beta_i\leq \frac{n\log r_{\min}-1}{n\log r_i}\leq \frac{\log r_{\min}-1}{\log r_i}\leq \frac{1+|\log r_{\min}|}{|\log r_{\max}|}=:K \qquad \forall i\in\{1,\dots,N\},
\]
by the second half of \eqref{eq:beta-sandwich}.
Recalling our convention that $0\log 0=0$, $\psi$ is continuous on $[0,K]^N$, hence it is uniformly continuous on this set. So there exists $\delta>0$ such that, whenever $|\beta_i-\beta_i'|<\delta$ for all $i$, we have $|\psi(\boldsymbol{\beta})-\psi(\boldsymbol{\beta}')|<\eps$. Let
\[
b:=\frac{\log r_{\min}}{\log r_{\max}}, \qquad \tilde{b}:=b\max_i|\log(Mp_i)|,
\]
and choose $n_0\in\NN$ so large that 
\begin{equation} \label{eq:n-must-be-large}
\frac{b}{n_0}<\delta, \qquad \frac{\tilde{b}}{n_0}<\eps \qquad\mbox{and} \qquad n_0\geq (N-1)|\log r_{\min}|.
\end{equation}
Given $n\geq n_0$ and $\boldsymbol{\beta}\in B_n'$, we construct a suitable vector $\boldsymbol{\beta}'$ as follows.
By the second half of \eqref{eq:beta-sandwich}, there is an index $j$ such that
\begin{equation} \label{eq:excluding-sum}
\sum_{i\neq j}\beta_i\log r_i>\frac{N-1}{N}\left(-1+\frac{\log r_{\min}}{n}\right)\geq -1,
\end{equation}
where the last inequality follows from \eqref{eq:n-must-be-large}.
Put $\beta_i'=\beta_i$ for $i\neq j$, and
\[
\beta_j':=-\frac{1+\sum_{i\neq j}\beta_i\log r_i}{\log r_j}.
\]
It follows at once that $h(\boldsymbol{\beta}')=1$, and $\beta_j'\geq 0$ by \eqref{eq:excluding-sum}. Furthermore,
\[
|\beta_j'-\beta_j|=\left|\frac{-1-\sum_i \beta_i\log r_i}{\log r_j}\right|\leq \frac{\log r_{\min}}{n\log r_{\max}}=\frac{b}{n}<\delta,
\]
where the first inequality follows from \eqref{eq:beta-sandwich}. Hence, $\psi(\boldsymbol{\beta})<\psi(\boldsymbol{\beta}')+\eps$. Finally, 
\begin{equation*}
\phi(\boldsymbol{\beta})\leq\phi(\boldsymbol{\beta}')+|\beta_j-\beta_j'||\log(Mp_j)|
\leq \phi(\boldsymbol{\beta}')+\frac{\tilde{b}}{n}<\phi(\boldsymbol{\beta}')+\eps,
\end{equation*}
and the Claim follows. $\diamond$

\bigskip

Combining the Claim with \eqref{eq:Gamma-bound} and noting that the factor $n^N\sqrt{n}$ is subexponential, we see that for each $\eps>0$ we can find $n_0\in\NN$ such that for all $n\geq n_0$,
\begin{align*}
\EE(\#\tilde{\Gamma}_n)&\leq e^{\eps n}\max\left\{\exp\left(n\big\{\psi(\boldsymbol{\beta})+\min\{0,\phi(\boldsymbol{\beta})\}\big\}\right): \boldsymbol{\beta}\in[0,\infty)^N, h(\boldsymbol{\beta})=1\right\}\\
&=e^{(s^*+\eps)n}.
\end{align*}
This completes the proof.
\end{proof}

\section{Proof of the lower bound} \label{sec:lower-bound}

In this section, we prove the second inequality in \eqref{eq:liminf-limsup-sandwich}:

\begin{proposition} \label{prop:key-lower-estimate}
We have
\[
\liminf_{n\to\infty}\frac{\log \EE(\#\tilde{\Gamma}_n)}{n}\geq s^*.
\]
\end{proposition}

\begin{proof}
Let $\boldsymbol{\beta}=(\beta_1,\dots,\beta_N)$ be any vector in $[0,\infty)^N$ such that $h(\boldsymbol{\beta})=1$. It suffices to show that
\[
\liminf_{n\to\infty}\frac{\log \EE(\#\tilde{\Gamma}_n)}{n}\geq \psi(\boldsymbol{\beta})+\min\{0,\phi(\boldsymbol{\beta})\}.
\]
By continuity, we may assume without loss of generality that $\beta_i>0$ for each $i$.

We divide the proof in a series of steps. For $\ii\in\Gamma_n$, define
\[
a_\ii:=\PP\left(\mbox{at least one path in $\CP_{|\ii|}$ is labeled $\ii$}\right), \qquad \ii\in\{1,2,\dots,N\}^*,
\]
so that
\[
\EE(\#\tilde{\Gamma}_n)=\sum_{\ii\in\Gamma_n} a_\ii.
\]
Unfortunately, the probabilities $a_\ii$ are difficult to compute or even estimate in general. The example at the end of Section \ref{sec:sketch} shows that they can, in general, be exponentially small. Our goal is to construct a subset $D_n$ of $\Gamma_n$ of words in which each digit $i\in\{1,\dots,N\}$ appears approximately $\beta_i n$ times, and then derive a lower estimate for $a_\ii$, $\ii\in D_n$.

\bigskip
{\em Step 1. (Construction of $D_n$)} Since $h(\boldsymbol{\beta})=-\sum \beta_i\log r_i=1$, we have for each $n\in\NN$,
\begin{equation} \label{eq:floor-inequalities}
\prod_{i=1}^N r_i^{1+\lfloor\beta_i n\rfloor}<e^{-n}=\left(\prod_{i=1}^N r_i^{\beta_i}\right)^n\leq \prod_{i=1}^N r_i^{\lfloor\beta_i n\rfloor}.
\end{equation}
Using these inequalities, we first construct a specific word in $\Gamma_n$ in which each digit $i$ occurs approximately $\beta_i n$ times. If equality holds in the rightmost inequality of \eqref{eq:floor-inequalities}, then the word
\[
\mathbf{v}:=1^{\lfloor\beta_1 n\rfloor}2^{\lfloor\beta_2 n\rfloor}\dots N^{\lfloor\beta_N n\rfloor}
\]
lies (barely) in $\Gamma_n$, where $i^k$ denotes the digit $i$ repeated $k$ times. Otherwise, this word is too short to belong to $\Gamma_n$, but by the left-most inequality in \eqref{eq:floor-inequalities}, we can extend $\mathbf{v}$ to a word in $\Gamma_n$ by appending the string $12\dots i_0$, for some $i_0\in\{1,2,\dots,N\}$. In either case, then, there is an integer $i_0\in\{0,1,\dots,N\}$ such that the word
\[
\ww:=1^{\lfloor\beta_1 n\rfloor}2^{\lfloor\beta_2 n\rfloor}\dots N^{\lfloor\beta_N n\rfloor}12\dots i_0.
\]
lies in $\Gamma_n$.
Let $L:=|\ww|=\sum_{i=1}^N\lfloor\beta_i n\rfloor+i_0$, and write $\ww=w_1\dots w_{L}$. Then
\begin{equation} \label{eq:L-bounds}
-N+\sum_{i=1}^N \beta_i n\leq L\leq \sum_{i=1}^N \beta_i n+N.
\end{equation}
Note furthermore, that all permutations of $\ww$ that fix the last digit $w_L$ also belong to $\Gamma_n$. Denote the set of all such permutations by $\hat{D}_n$. It would be easier to estimate the probability $a_\ii$, for $\ii\in\hat{D}_n$, if the occurrences of each digit were somewhat evenly distributed across the word $\ii$. Therefore, we form a subset $D_n$ of $\hat{D}_n$ by imposing an additional requirement.

For each $i\in\{1,\dots,N\}$, let $t_i$ be the number of occurrences of the digit $i$ in $w_1\dots w_{L-1}$. Then $t_i$ is either $\lfloor\beta_i n\rfloor$, $\lfloor\beta_i n\rfloor-1$, or $\lfloor\beta_i n\rfloor+1$, and $\sum_{i=1}^N t_i=L-1$. Put $m:=\lfloor\sqrt{L-1}\rfloor$, and assume $n$ (and hence $L$) is large enough so that $m\gg N$. Put
\[
k_i:=\left\lfloor\frac{t_i}{m}\right\rfloor, \qquad i=1,\dots,N,
\]
and
\[
q:=\sum_{i=1}^N k_i.
\]
Note that $L-1-mN<mq\leq L-1$. Write $L-1=mq+r$, so $0\leq r<mN$. Now let $D_n$ be the set of all words in $\hat{D}_n$ for which each of the first $m$ blocks of length $q$ contains the digit $i$ exactly $k_i$ times, for $i=1,\dots,N$. Precisely, $\ii=i_1\dots i_L\in D_n$ if and only if $\ii\in\hat{D}_n$ and for each $i\in\{1,\dots,N\}$ and $l\in\{1,\dots,m\}$,
\[
\#\{\nu\in\{(l-1)q+1,\dots,lq\}: i_\nu=i\}=k_i.
\]

\bigskip
{\em Step 2. (Estimating the cardinality of $D_n$)} We claim that
\begin{equation} \label{eq:D_n-lower-estimate}
\liminf_{n\to\infty}\frac{\log\#D_n}{n}\geq \psi(\boldsymbol{\beta}).
\end{equation}
To see this, note first that $\psi$ satisfies the positive homogeneity property
\begin{equation} \label{eq:psi-homogeneous}
\psi(\alpha\xx)=\alpha\psi(\xx) \qquad\mbox{for any $\alpha>0$ and $\xx\in[0,\infty)^N$}.
\end{equation}
We use Stirling's approximation to get the initial estimate
\begin{align*}
\#D_n\geq {\binom{q}{k_1,\dots,k_N}}^m&=\left(\frac{\sqrt{2\pi q}\cdot q^q}{\prod\sqrt{2\pi k_i}\cdot \prod k_i^{k_i}}
\cdot (1+o(1))\right)^m\\
& \geq \left(\frac{q^q}{(2\pi)^{N/2}q^{N/2}\prod k_i^{k_i}}\cdot (1+o(1))\right)^m.
\end{align*}
Thus,
\begin{align*}
\frac{\log\#D_n}{m}&\geq q\log q-\frac{N}{2}\log 2\pi-\frac{N}{2}\log q-\sum_i k_i\log k_i+o(1)\\
&=q\log q-\sum_i k_i\log k_i+O(\log q). 
\end{align*}
Here we want to replace $k_i$ with $k_i':=\beta_i n/m$ (which is not necessarily an integer), and similarly, replace $q$ with $q':=\sum_i k_i'$.
The definitions of $t_i$ and $k_i$ imply
\begin{equation} \label{eq:k-i-estimates}
\frac{\beta_i n-2}{m}-1<\frac{t_i}{m}-1<k_i\leq\frac{t_i}{m}\leq \frac{\beta_i n+1}{m},
\end{equation}
so certainly $|k_i-k_i'|\leq 3$. The mean value theorem then gives $k_i\log k_i-k_i'\log k_i'=O(\log k_i)=O(\log q)$. Similarly, $|q-q'|\leq 3N$ and so $q\log q-q'\log q'=O(\log q)$.
Observe that 
\[
q\leq \frac{L}{m}\sim \sqrt{L}\sim \sqrt{\sum \beta_i n}=O(\sqrt{n})
\]
by \eqref{eq:L-bounds}. 
Hence,
\begin{align*}
\frac{\log\#D_n}{m}&\geq q'\log q'-\sum_i k_i'\log k_i'+O(\log q)\\
&=\psi\left(\frac{\beta_1 n}{m},\dots,\frac{\beta_N n}{m}\right)+O(\log \sqrt{n})\\
&=\frac{n}{m}\psi(\boldsymbol{\beta})+O(\log n),
\end{align*}
where in the last step we used \eqref{eq:psi-homogeneous}. Multiplying by $m/n$, we obtain \eqref{eq:D_n-lower-estimate}.

\bigskip
{\em Step 3. (Estimating $a_\ii$ for $\ii\in D_n$)} Recall from Step 1 that all words in $D_n$ have length $L$. Note that $a_\ii$ satisfies the recursion
\begin{equation} \label{eq:recursion}
a_{i_1i_2\dots i_L}=1-(1-p_{i_1}a_{i_2\dots i_L})^M.
\end{equation}
This can be seen by a branching argument: Let $G_\ii$ be the event that {\em no} path of length $n$ starting at the root of the  tree has label sequence $\mathbf{i}$. For $k=1,\dots,M$, let $G_{\ii,k}$ be the event that no path starting with the $k$th edge emanating from the root has label sequence $\mathbf{i}$. By independence, $\PP(G_\ii)=\prod_{k=1}^M \PP(G_{\ii,k})$. Now $G_{\ii,k}^c$ happens if and only if the $k$th edge below the root receives label $i_1$ and the label sequence $a_{i_2\dots i_L}$ occurs at least once in the tree starting from the $k$th child of the root. By the stationarity of the labeling process, the latter event has probability $a_{i_2\dots i_L}$. Thus,
\[
a_\ii=1-\PP(G_\ii)=1-\prod_{k=1}^M\big(1-\PP(G_{\ii,k}^c)\big)=1-(1-p_{i_1}a_{i_2\dots i_L})^M.
\]

Let
\[
p_{\min}:=\min_i p_i, \qquad\mbox{and} \qquad p_{\max}:=\max_i p_i.
\]
We assume that $p_{\min}>0$. This is not a real restriction: If $p_i=0$ for some $i$, we simply remove the map $f_i$ from the IFS and reduce the value of $N$. 
 Since the function $x\mapsto 1-(1-p_{\min}x)^M$ is increasing and concave on $[0,1]$, it follows from \eqref{eq:recursion} that
\begin{equation} \label{eq:lower-linear}
a_{i_1i_2\dots i_L}\geq \gamma a_{i_2\dots i_L} \qquad\forall i_1,\dots,i_L,
\end{equation}
where $\gamma:=1-(1-p_{\min})^M>0$ is a universal constant. Furthermore, it follows from \eqref{eq:recursion} and Bernoulli's inequality that
\begin{equation} \label{eq:upper-linear}
a_{i_1i_2\dots i_L}\leq Mp_{i_1}a_{i_2\dots i_L}
\end{equation}
and
\begin{align} 
\begin{split}
a_{i_1i_2\dots i_L}&\geq Mp_{i_1}a_{i_2\dots i_L}-\binom{M}{2}p_{i_1}^2 a_{i_2\dots i_L}^2\\
&\geq Mp_{i_1}a_{i_2\dots i_L}(1-Ma_{i_2\dots i_L}).
\end{split}
\label{eq:quadratic-lower}
\end{align}
Recall that $L=mq+r+1$, where $m=\lfloor\sqrt{L-1}\rfloor$, $q=\sum_i k_i$ and $0\leq r<mN$. Our goal is to show that the probability $a_{i_j\dots i_L}$ cannot oscillate too much as $j$ moves backwards from $L$ to $1$. In view of the definition of $D_n$, we consider decrementing $j$ in steps of size $q$. Precisely, we always take $j$ to be one more than a multiple of $q$. To shorten notation, put
\[
b_l(\ii):=a_{i_{lq+1}\dots i_L}, \qquad \ii\in D_n, \quad l=0,1,\dots,m.
\]
First, \eqref{eq:lower-linear} implies
\begin{equation} \label{eq:remainder-block}
b_m(\ii)\geq \gamma^{L-mq}=\gamma^{r+1}\geq \gamma^{mN}.
\end{equation} 
Fix $\eta>0$. Take $l\in\{1,\dots,m\}$. We consider two cases.
If
\[
b_l(\ii)>\eta(Mp_{\max})^{-q},
\]
then by \eqref{eq:lower-linear},
\[
b_{l-1}(\ii)\geq \gamma^q b_l(\ii)>\eta\left(\frac{\gamma}{Mp_{\max}}\right)^q.
\]
Suppose instead that
\begin{equation} \label{eq:small-intermediate}
b_l(\ii)\leq\eta(Mp_{\max})^{-q}.
\end{equation}
Then $a_{i_j\dots i_L}\leq\eta$ for $(l-1)q<j\leq lq$ by \eqref{eq:upper-linear}, and so \eqref{eq:quadratic-lower} implies
\begin{equation} \label{eq:block-lower-inequality}
b_{l-1}(\ii)\geq \prod_{i=1}^N \big(Mp_i(1-\eta M)\big)^{k_i}b_l(\ii)
=\big(M(1-\eta M)\big)^q\left(\prod_{i=1}^N p_i^{k_i}\right)b_l(\ii).
\end{equation}
At this point it is convenient to define (as in the proof of the upper bound)
\[
\alpha_i:=\frac{\beta_i}{\beta_1+\dots+\beta_N}, \qquad i=1,\dots,N.
\]
From the left half of \eqref{eq:k-i-estimates} it follows that
\[
\frac{\sum_i\beta_i n}{m}<q+N+\frac{N}{m}.
\]
Hence, by the right half of \eqref{eq:k-i-estimates},
\[
k_i\leq \frac{\beta_i n+1}{m}=\frac{\alpha_i\sum_j\beta_j n+1}{m}<\alpha_i\left(q+N+\frac{N}{m}\right)+\frac{1}{m}<\alpha_i(q+2N)+1,
\]
so that
\[
\prod_{i=1}^N p_i^{k_i}\geq \prod_{i=1}^N p_i^{\alpha_i(q+2N)+1}\geq \prod_{i=1}^N p_i^{\alpha_i q}\cdot \prod_{i=1}^N p_{\min}^{2N\alpha_i+1}=p_{\min}^{3N}\prod_{i=1}^N p_i^{\alpha_i q}=:C_{\pp}\prod_{i=1}^N p_i^{\alpha_i q}.
\]
Putting this estimate back into \eqref{eq:block-lower-inequality} yields
\[
b_{l-1}(\ii)\geq C_{\pp}\big(M(1-\eta M)\big)^q \left(\prod_{i=1}^N p_i^{\alpha_i q}\right)b_l(\ii)
\geq C_{\pp}(1-\eta M)^q K_{\vec\alpha}^q\, b_l(\ii),
\]
where
\[
K_{\vec\alpha}:=\min\left\{1,M\prod_{i=1}^N p_i^{\alpha_i}\right\}=\min\left\{1,\prod_{i=1}^N (Mp_i)^{\alpha_i}\right\}.
\]
(We do the estimate this way, because it will later be important that $K_{\vec\alpha}\leq 1$.) Now let
\[
u(\ii):=\#\{l\in\{1,\dots,m\}:\ \mbox{\eqref{eq:small-intermediate} holds}\}.
\]
Then for each $\ii\in D_n$ we obtain, using that $u(\ii)\leq m$ and $u(\ii)q\leq mq\leq L$,
\begin{align*}
a_\ii&\geq\eta\left(\frac{\gamma}{Mp_{\max}}\right)^q(1-\eta M)^{u(\ii)q}C_{\pp}^{u(\ii)} K_{\vec\alpha}^{u(\ii)q}\, b_m(\ii)\\
&\geq \eta\left(\frac{\gamma}{Mp_{\max}}\right)^q(1-\eta M)^{mq}C_{\pp}^m K_{\vec\alpha}^{mq}\gamma^{mN}\\
&\geq \eta C_{\pp}^m\gamma^{mN}\left(\frac{\gamma}{Mp_{\max}}\right)^q(1-\eta M)^L K_{\vec\alpha}^L, 
\end{align*}
where we also used \eqref{eq:remainder-block}. Note that this bound does not depend on $\ii$. Since $m\leq \sqrt{L}$ and $q\leq L/m\sim\sqrt{L}$, the factors $C_{\pp}^m\gamma^{mN}$ and $\left(\frac{\gamma}{Mp_{\max}}\right)^q$ go to zero sub-exponentially (if at all), hence for all large enough $n$, we have
\begin{equation*} 
\min_{\ii\in D_n}a_\ii\geq (1-\eta M)^L (K_{\vec\alpha}-\eta)^L\geq \big\{(1-\eta M)(K_{\vec\alpha}-\eta)\big\}^{\sum_i\beta_i n+N},
\end{equation*}
by \eqref{eq:L-bounds}. Since $\eta$ was arbitrary and $N$ is constant, it follows that
\begin{equation} \label{eq:a_i-estimate}
\liminf_{n\to\infty}\frac{1}{n}\log\left(\min_{\ii\in D_n}a_\ii\right)\geq \sum_i\beta_i \log K_{\vec\alpha}.
\end{equation}

\medskip
{\em Step 4. (Conclusion)} Clearly, we have
\[
\sum_{\ii\in\Gamma_n}a_\ii\geq \sum_{\ii\in D_n}a_\ii \geq \#D_n\cdot \min_{\ii\in D_n}a_\ii.
\]
Thus, by \eqref{eq:D_n-lower-estimate} and \eqref{eq:a_i-estimate},
\begin{align*}
\liminf_{n\to\infty} \frac{1}{n}\log\left(\sum_{\ii\in\Gamma_n}a_\ii\right)&\geq \liminf_{n\to\infty} \frac{\log\#D_n}{n}+\liminf_{n\to\infty} \frac{\log \min_{\ii\in D_n}a_\ii}{n}\\
&\geq \psi(\boldsymbol{\beta})+\sum_{i=1}^N \beta_i\log K_{\vec\alpha}.
\end{align*}
Recalling that $\alpha_i=\beta_i/(\beta_1+\dots+\beta_N)$ and noting that
\begin{align*}
\sum_{i=1}^N \beta_i\log K_{\vec\alpha}&=\sum_{i=1}^N\beta_i\cdot\min\left\{0,\sum_i \alpha_i\log(Mp_i)\right\}\\
&=\min\left\{0,\sum_i\beta_i\log(Mp_i)\right\}=\min\big\{0,\phi(\boldsymbol{\beta})\big\},
\end{align*}
we complete the proof.
\end{proof}

\section{Solving the constrained optimization} \label{sec:main-proof}

We first establish the existence and uniqueness of the pair $(\hat{s},\hat{t})$ in Theorem \ref{thm:main} (iii).

\begin{lemma} \label{lem:s-t-existence}
Assume
\begin{equation} \label{eq:main-case-re}
\sum_{i=1}^N r_i^{s_0}\log(Mp_i)<0<\sum_{i=1}^N p_i r_i^{\tilde{s}}\log(Mp_i).
\end{equation}
Then there is a unique pair $(\hat{s},\hat{t})$ of real numbers, with $\hat{s}>0$, satisfying the equations
\begin{gather} 
\sum_{i=1}^N r_i^{\hat{s}}(Mp_i)^{\hat{t}}=1, \label{eq:s-t-equation1} \\
\sum_{i=1}^N r_i^{\hat{s}}(Mp_i)^{\hat{t}}\log(Mp_i)=0.  \label{eq:s-t-equation2}
\end{gather}
\end{lemma}

\begin{proof}
Define the functions
\[
G_1(s,t):=\sum_{i=1}^N r_i^s (Mp_i)^t, \qquad G_2(s,t):=\sum_{i=1}^N r_i^s (Mp_i)^t\log(Mp_i).
\]
Note that for fixed $t$, $G_1(s,t)$ is strictly decreasing in $s$, with $\lim_{s\to\infty}G_1(s,t)=0$ and $\lim_{s\to-\infty}G_1(s,t)=\infty$. Thus, for fixed $t$, the equation $G_1(s,t)=1$ has a unique solution $s=g(t)$, say.  
Differentiating implicitly yields
\begin{equation} \label{eq:first-derivative}
g'(t)=-\frac{\sum_i r_i^s(Mp_i)^t\log(Mp_i)}{\sum_i r_i^s(Mp_i)^t\log r_i}.
\end{equation}
This shows that $G_2(s,t)=0$ precisely at those points $(s,t)$ where $s=g(t)$ and $g'(t)=0$. In other words, the curves
\[
C_1: G_1(s,t)=1 \qquad\mbox{and} \qquad C_2: G_2(s,t)=0
\]
intersect each other at those and only those points where $C_1$ has a horizontal tangent line.

Next, we rewrite \eqref{eq:first-derivative} as
\[
0=\sum_{i=1}^N r_i^{g(t)}(Mp_i)^t\{g^\prime(t)\log r_i+\log(Mp_i)\}.
\]
Differentiating implicitly once more, we find
\[
0=\sum_{i=1}^N r_i^{g(t)}(Mp_i)^t\left[g^{\prime \prime}(t)\log r_i+\{g^\prime(t)\log r_i+\log(Mp_i)\}^2\right].
\]
(These equations are well known in the multifractal analysis of self-similar sets.) It follows that $g''(t)\geq 0$, and so $g$ is convex. In fact, it is strictly convex unless $Mp_i=r_i^\lambda$ for all $i$ and some constant $\lambda$. But we can rule out the latter possibility because this would imply that $\log(Mp_i)$ has constant sign, contradicting \eqref{eq:main-case}. Hence, there is at most one point where $g'(t)=0$, and so $C_1$ and $C_2$ intersect at most once.

In order to show that $C_1$ and $C_2$ do intersect, we need to show that $g'(t)=0$ for some $t$. Observe that the equation $\sum_i p_i r_i^{\tilde s}=1/M$ is equivalent to $G_1(\tilde{s},1)=1$, so $g(1)=\tilde{s}$, and hence
\[
\sum_i r_i^{g(1)}(Mp_i)^1\log(Mp_i)=\sum_i Mp_i r_i^{\tilde s}\log(Mp_i)>0
\]
by the second inequality in \eqref{eq:main-case}. This implies $g'(1)>0$ by \eqref{eq:first-derivative}. On the other hand, the equation $\sum_i r_i^{s_0}=1$ is equivalent to $G_1(s_0,0)=1$, so $g(0)=s_0$ and hence,
\[
\sum_i r_i^{g(0)}(Mp_i)^0\log(Mp_i)=\sum_i r_i^{s_0}\log(Mp_i)<0
\]
by the first inequality in \eqref{eq:main-case}. This implies $g'(0)<0$. Therefore, there exists $\hat{t}\in(0,1)$ such that $g'(\hat{t})=0$, as required. 

It remains to verify that $g(\hat{t})>0$. By the second inequality in \eqref{eq:main-case}, $Mp_{i_0}>1$ for at least one $i_0$, and since $\hat{t}>0$, we have $(Mp_{i_0})^{\hat t}>1$.  By the definition of $g$, we thus have
\[
\sum_i r_i^{g(\hat{t})}(Mp_i)^{\hat t}=1<\sum_i(Mp_i)^{\hat t},
\]
and this implies $g(\hat{t})>0$.
\end{proof}

\begin{proof}[Proof of Theorem \ref{thm:main}]
(i) Assume 
\begin{equation} \label{eq:M-large-re}
\sum_{i=1}^N r_i^{s_0}\log(Mp_i)\geq 0,
\end{equation}
and recall that 
\begin{gather}
\psi(x_1,\dots,x_N):=\sum_{i=1}^N x_i\left(\log\sum_{j=1}^N x_j-\log x_i\right), \\
\phi(x_1,\dots,x_N):=\sum_{i=1}^N x_i\log(Mp_i), \\
h(x_1,\dots,x_N):=-\sum_{i=1}^N x_i\log r_i.
\end{gather}
Put $\boldsymbol\beta=(\beta_1,\dots,\beta_N)$, where
\begin{equation} \label{eq:case-1-optimal-beta}
\beta_i:=-\frac{r_i^s}{\sum_j r_j^s\log r_j}, \qquad i=1,\dots,N.
\end{equation}
It is immediate that $h(\boldsymbol{\beta})=1$, while \eqref{eq:M-large-re} implies $\phi(\boldsymbol{\beta})=\sum_i\beta_i\log(Mp_i)\geq 0$. Furthermore,
\[
\psi(r_1^{s_0},\dots,r_N^{s_0})=-s_0\sum_i r_i^{s_0}\log r_i,
\]
since $\sum_i r_i^{s_0}=1$. Thus, since $\psi(\alpha\xx)=\alpha\psi(\xx)$ for all $\alpha>0$ and $\xx\in[0,\infty)^N$ (cf. \eqref{eq:psi-homogeneous}),
\[
\psi(\boldsymbol{\beta})=-\frac{1}{\sum_j r_j^{s_0}\log r_j}\psi(r_1^{s_0},\dots,r_N^{s_0})=s_0.
\]
Applying Theorem \ref{thm:variational} we conclude $\dim_H F\geq s_0$ almost surely. Since $F \subset \Lambda$, the reverse inequality holds trivially. Thus, $\dim_H F=s_0$ almost surely.


(ii) Next, assume 
\begin{equation} \label{eq:M-small-re}
\sum_{i=1}^N p_i r_i^{\tilde{s}}\log(Mp_i)\leq 0,
\end{equation}
where $\tilde{s}$ satisfies $\sum_i p_i r_i^{\tilde s}=1/M$. Put
\begin{equation} \label{eq:case-2-optimal-beta}
\beta_i':=-\frac{p_i r_i^{\tilde{s}}}{\sum_j p_j r_j^{\tilde s}\log r_j}, \qquad i=1,\dots,N,
\end{equation}
and $\boldsymbol{\beta}':=(\beta_1',\dots,\beta_N')$. Clearly $h(\boldsymbol{\beta}')=1$, and \eqref{eq:M-small-re} implies $\phi(\boldsymbol{\beta}')\leq 0$. Note that
\begin{align*}
\psi\big(Mp_1 r_1^{\tilde s},\dots, Mp_N r_N^{\tilde s}\big)&=-\sum_i Mp_i r_i^{\tilde s}\log(Mp_i r_i^{\tilde s})\\
&=-\sum_i Mp_i r_i^{\tilde s}\big(\log(Mp_i)+\tilde{s}\log r_i\big).
\end{align*}
So by \eqref{eq:psi-homogeneous},
\[
\dim_H F\geq \psi(\boldsymbol{\beta}')+\phi(\boldsymbol{\beta}')=\frac{\sum_i Mp_i r_i^{\tilde s}\big(\log(Mp_i)+\tilde{s}\log r_i\big)}{\sum_j Mp_j r_j^{\tilde s}\log r_j}-\frac{\sum_i p_i r_i^{\tilde s}\log(Mp_i)}{\sum_j p_j r_j^{\tilde s}\log r_j}=\tilde{s}.
\]
The reverse inequality follows since the vector $\boldsymbol{\beta}'$ defined by \eqref{eq:case-2-optimal-beta} maximizes $\psi(\xx)+\phi(\xx)$ subject to the constraint $h(\xx)=1$, as is easily seen by a Lagrange multiplier computation.
Hence, $\dim_H F=\tilde{s}\leq s_0$ almost surely.

Finally, we show that $\tilde{s}=s_0$ if and only if $N=M$ and $p_i=1/N$ for all $i$. Let $\boldsymbol{\beta}$ again be the vector defined by \eqref{eq:case-1-optimal-beta}. If $\boldsymbol{\beta}\neq\boldsymbol{\beta}'$, then 
\[
s_0=\psi(\boldsymbol{\beta})>\psi(\boldsymbol{\beta}')\geq \psi(\boldsymbol{\beta}')+\phi(\boldsymbol{\beta}')=\tilde{s},
\]
since $\boldsymbol{\beta}$ is the unique maximizer of $\psi(\xx)$ subject to $h(\xx)=1$. On the other hand, if $\boldsymbol{\beta}=\boldsymbol{\beta}'$, then it follows from equating the right hand sides of \eqref{eq:case-1-optimal-beta} and \eqref{eq:case-2-optimal-beta} that $p_i$ is constant in $i$, and so $p_i=1/N$ for $i=1,\dots,N$. Hence by \eqref{eq:weighted-Moran}, $\sum_i r_i^{\tilde{s}}=N/M$. So if $\tilde{s}=s_0$, it must be the case that $N=M$ and $p_i=1/N$ for all $i$. The reverse implication is obvious.

(iii) Assume \eqref{eq:main-case-re}. The existence and uniqueness of $(\hat{s},\hat{t})$ were established in Lemma \ref{lem:s-t-existence}. Let $\boldsymbol{\beta}$ and $\boldsymbol{\beta}'$ be the vectors from \eqref{eq:case-1-optimal-beta} and \eqref{eq:case-2-optimal-beta}, respectively. Since $\psi$ is unimodal on the hyperplane $h(\xx)=1$ with its constrained maximum at $\boldsymbol{\beta}$ and $\phi(\boldsymbol{\beta})<0$ by the first inequality in \eqref{eq:main-case}, it follows that 
\begin{align*}
\max\big\{\psi(\xx)&+\min\{0,\phi(\xx)\}: \xx\in[0,\infty)^N, h(\xx)=1, \phi(\xx)\geq 0\big\}\\
&=\max\big\{\psi(\xx): \xx\in[0,\infty)^N, h(\xx)=1, \phi(\xx)\geq 0\big\}\\
&=\max\big\{\psi(\xx): \xx\in[0,\infty)^N, h(\xx)=1, \phi(\xx)=0\big\}\\
&=\max\big\{\psi(\xx)+\min\{0,\phi(\xx)\}: \xx\in[0,\infty)^N, h(\xx)=1, \phi(\xx)=0\big\}.
\end{align*}
Likewise, since $\psi+\phi$ is unimodal on $\{\xx:h(\xx)=1\}$ with its constrained maximum at $\boldsymbol{\beta}'$ and $\phi(\boldsymbol{\beta}')>0$ by the second inequality in \eqref{eq:main-case}, we have
\begin{align*}
\max\big\{\psi(\xx)&+\min\{0,\phi(\xx)\}: \xx\in[0,\infty)^N, h(\xx)=1, \phi(\xx)\leq 0\big\}\\
&=\max\big\{\psi(\xx)+\phi(\xx): \xx\in[0,\infty)^N, h(\xx)=1, \phi(\xx)\leq 0\big\}\\
&=\max\big\{\psi(\xx)+\phi(\xx): \xx\in[0,\infty)^N, h(\xx)=1, \phi(\xx)=0\big\}\\
&=\max\big\{\psi(\xx)+\min\{0,\phi(\xx)\}: \xx\in[0,\infty)^N, h(\xx)=1, \phi(\xx)=0\big\}.
\end{align*}
Thus, the constrained maximum of $\psi(\xx)+\min\{0,\phi(\xx)\}$ over $\{\xx:h(\xx)=1\}$ must be attained on the hyperplane $\phi(\xx)=0$. Solving the Lagrange multiplier problem 
\[
\mbox{maximize}\ \psi(\xx)\qquad \mbox{subject to}\qquad h(\xx)=1, \quad \phi(\xx)=0
\]
yields that the maximum is attained (uniquely) at the point $\hat{\boldsymbol\beta}=(\hat{\beta}_1,\dots,\hat{\beta}_N)$ given by
\[
\hat{\beta}_i=-\frac{r_i^{\hat s}(Mp_i)^{\hat t}}{\sum_j r_j^{\hat s}(Mp_j)^{\hat t}\log r_j}, \qquad i=1,\dots,N.
\]
We then have
\begin{align*}
\dim_H F&=\psi\big(\hat{\boldsymbol\beta}\big)=-\frac{\psi\big(r_1^{\hat s}(Mp_1)^{\hat t},\dots,r_N^{\hat s}(Mp_N)^{\hat t}\big)}{\sum_i r_i^{\hat s}(Mp_i)^{\hat t}\log r_i}\\
&=\frac{\sum_i r_i^{\hat s}(Mp_i)^{\hat t}\log\big(r_i^{\hat s}(Mp_i)^{\hat t}\big)}{\sum_i r_i^{\hat s}(Mp_i)^{\hat t}\log r_i}\\
&=\frac{\sum_i r_i^{\hat s}(Mp_i)^{\hat t}\big\{\hat{s}\log r_i+\hat{t}\log(Mp_i)\big\}}{\sum_i r_i^{\hat s}(Mp_i)^{\hat t}\log r_i}=\hat{s},
\end{align*}
where the first equality used \eqref{eq:psi-homogeneous} and \eqref{eq:s-t-equation1}, and the last inequality follows from \eqref{eq:s-t-equation2}. Finally, that $\hat{s}<s_0$ follows since $\hat{\boldsymbol\beta}\neq\boldsymbol{\beta}$ (because $\phi(\hat{\boldsymbol\beta})=0\neq\phi(\boldsymbol{\beta})$), and so $\hat{s}=\psi(\hat{\boldsymbol\beta})<\psi(\boldsymbol{\beta})=s_0$.
\end{proof}

\section{Proofs of the special cases} \label{sec:corollaries}

\begin{proof}[Proof of Corollary \ref{cor:homogeneous}]
Suppose $r_i=r$ for $i=1,\dots,N$. Then \eqref{eq:Moran} and \eqref{eq:weighted-Moran} give
\[
s_0=-\frac{\log N}{\log r}, \qquad \tilde{s}=-\frac{\log M}{\log r}.
\]
Furthermore, the conditions \eqref{eq:M-large-condition} and \eqref{eq:M-small-condition} simplify to $M\geq U(N,\pp)$ and $M\leq L(N,\pp)$, respectively. Finally, for $L(N,\pp)\leq M\leq U(N,\pp)$, \eqref{eq:s-t-equation-a} becomes
\begin{equation} \label{eq:simpler-s-t-equation-a}
r^{\hat s}\sum_i (Mp_i)^{\hat t}=1,
\end{equation}
and \eqref{eq:s-t-equation-b} reduces to
\[
\sum_i p_i^{\hat t}\log(Mp_i)=0.
\]
Putting $\lambda:=\hat{t}$, we see that $\lambda$ satisfies \eqref{eq:mu-equation}, and then \eqref{eq:simpler-s-t-equation-a} gives the desired expression for $\dim_H F$.
\end{proof}

\begin{proof}[Proof of Theorem \ref{cor:two-maps-homogeneous}]
In principle, it is possible (but cumbersome!) to derive Theorem \ref{cor:two-maps-homogeneous} directly from Corollary \ref{cor:homogeneous}. It is simpler, however, to deduce it from Theorem \ref{thm:variational}. For $N=M=2$, we are always in the case $L(N,\pp)\leq M\leq U(N,\pp)$, so the constrained maximum in Theorem \ref{thm:variational} is attained on the line 
\[
\phi(x_1,x_2)=x_1\log(2p)+x_2\log(2(1-p))=0.
\]
Together with the equation $h(x_1,x_2)=-(x_1+x_2)\log r=1$, this gives a unique point $(x_1,x_2)$ which satisfies
\[
\frac{x_1}{x_1+x_2}=-\frac{\log(2(1-p))}{\log p-\log(1-p)}=1-\xi, \qquad \frac{x_2}{x_1+x_2}=\frac{\log(2p)}{\log p-\log(1-p)}=\xi.
\]
Thus,
\begin{align*}
\dim_H F&=\psi(x_1,x_2)=(x_1+x_2)\psi\left(\frac{x_1}{x_1+x_2},\frac{x_2}{x_1+x_2}\right)\\
&=-\frac{\psi(1-\xi,\xi)}{\log r}=\frac{\xi\log\xi+(1-\xi)\log(1-\xi)}{\log r},
\end{align*}
as desired.
\end{proof}

The proof of Theorem \ref{cor:two-maps-nonhomogeneous} is more subtle; we first develop a few lemmas. Recall that $\tilde{s}=\tilde{s}(p)$ is the solution of the equation
\begin{equation} \label{eq:s-tilde-equation-bis}
pr_1^{\tilde s}+(1-p)r_2^{\tilde s}=\frac12.
\end{equation}

\begin{lemma} \label{lem:negative-difference}
Assume $r_1<r_2$ and $1/2\leq p<1$. Then
\[
r_1^{\tilde s}\log r_1<r_2^{\tilde s}\log r_2.
\]
\end{lemma}

\begin{proof}
Note that the inequality is not trivial because the function $r\mapsto r^{\tilde s}\log r$ is not monotone on $[0,1]$. However, the inequality is equivalent to
\[
\left(\frac{r_1}{r_2}\right)^{\tilde s}> \frac{\log r_2}{\log r_1},
\]
and since $\tilde{s}\leq s_0$, it suffices therefore to show that
\begin{equation} \label{eq:s_0-inequality}
r_1^{s_0}\log r_1<r_2^{s_0}\log r_2.
\end{equation}
But observe that
\[
r_1^{s_0}\log r_1-r_2^{s_0}\log r_2=\frac{1}{s_0}\left[r_1^{s_0}\log r_1^{s_0}-r_2^{s_0}\log r_2^{s_0}\right]=\frac{1}{s_0}f(r_1^{s_0}),
\]
where $f(x):=x\log x-(1-x)\log(1-x)$, and it is easy to see that $f$ is convex on $[0,1/2]$ with $f(0)=f(1/2)=0$, hence $f(x)\leq 0$ on $[0,1/2]$. Since $r_1<r_2$ and $r_1^{s_0}+r_2^{s_0}=1$, it follows that $r_1^{s_0}<1/2$. Hence, we have \eqref{eq:s_0-inequality}.
\end{proof}

\begin{lemma} \label{lem:convex}
Let $r_1<r_2$. Then the function $\tilde{s}(p)$ is strictly decreasing and strictly convex on $1/2\leq p\leq 1$.
\end{lemma}

\begin{proof}
Differentiating \eqref{eq:s-tilde-equation-bis} implicitly with respect to $p$ gives
\begin{equation} \label{eq:derivative-of-s-tilde}
\left\{pr_1^{\tilde s}\log r_1+(1-p)r_2^{\tilde s}\log r_2\right\}\frac{d\tilde{s}}{dp}+(r_1^{\tilde s}-r_2^{\tilde s})=0,
\end{equation}
which shows that $d\tilde{s}/dp<0$. Differentiating a second time results in
\begin{align}
\begin{split}
&\left\{pr_1^{\tilde s}\log r_1+(1-p)r_2^{\tilde s}\log r_2\right\}\frac{d^2\tilde{s}}{dp^2}\\
&\qquad +\left\{pr_1^{\tilde s}(\log r_1)^2+(1-p)r_2^{\tilde s}(\log r_2)^2\right\}\left(\frac{d\tilde{s}}{dp}\right)^2\\
&\qquad +2\left(r_1^{\tilde s}\log r_1-r_2^{\tilde s}\log r_2\right)\frac{d\tilde{s}}{dp}=0.
\end{split}
\label{eq:second-derivative-of-s-tilde}
\end{align}
Here the second term is clearly positive. The third term is also positive by Lemma \ref{lem:negative-difference} and $d\tilde{s}/dp<0$. Therefore, the first term must be negative, and this implies $d^2\tilde{s}/dp^2>0$.
\end{proof}

\begin{lemma} \label{lem:g-tilde}
Define the function
\begin{equation} \label{eq:g-tilde}
\tilde{g}(p):=pr_1^{\tilde s(p)}\log(2p)+(1-p)r_2^{\tilde s(p)}\log(2(1-p)),
\end{equation}
where $\tilde{s}(p)$ is given by \eqref{eq:s-tilde-equation-bis}. 
Assume $r_1<r_2$. Then $\tilde{g}$ is strictly convex on $[1/2,1]$. Furthermore, $\tilde{g}(1/2)=0$, $\tilde{g}'(1/2)<0$ and $\tilde{g}(1)>0$. Hence, $\tilde{g}$ has a unique zero in the interval $(1/2,1)$.
\end{lemma}

\begin{proof}
By \eqref{eq:s-tilde-equation-bis} we can simplify $\tilde{g}(p)$ somewhat to
\begin{equation} \label{eq:g-tilde-simpplified}
\tilde{g}(p)=pr_1^{\tilde s(p)}\log p+(1-p)r_2^{\tilde s(p)}\log(1-p)+\frac12\log 2.
\end{equation}
Differentiating twice gives
\begin{align}
\begin{split}
\tilde{g}'(p)&=\left\{p\log p\cdot r_1^{\tilde s}\log r_1+(1-p)\log(1-p)\cdot r_2^{\tilde s}\log r_2\right\}\frac{d\tilde{s}}{dp}\\
&\qquad +r_1^{\tilde s}\big(\log p+1\big)-r_2^{\tilde s}\big(\log(1-p)+1\big)
\end{split}
\label{eq:g-tilde-prime}
\end{align}
(where we suppress the dependence of $\tilde{s}$ on $p$ to avoid clutter), and, after some simplification,
\begin{align}
\begin{split}
\tilde{g}''(p)&=\left\{p\log p\cdot r_1^{\tilde s}\log r_1+(1-p)\log(1-p)\cdot r_2^{\tilde s}\log r_2\right\}\frac{d^2\tilde{s}}{dp^2}\\
&\qquad +\left\{p\log p\cdot r_1^{\tilde s}(\log r_1)^2+(1-p)\log(1-p)\cdot r_2^{\tilde s}(\log r_2)^2\right\}\left(\frac{d\tilde{s}}{dp}\right)^2\\
&\qquad +2\left\{(\log p+1)\cdot r_1^{\tilde s}\log r_1-\big(\log(1-p)+1\big)\cdot r_2^{\tilde s}\log r_2\right\}\frac{d\tilde{s}}{dp}\\
&\qquad +\frac{r_1^{\tilde s}}{p}+\frac{r_2^{\tilde s}}{1-p}.
\end{split}
\label{eq:g-tilde-double-prime}
\end{align}
It is not immediately clear whether this is positive. But if we multiply \eqref{eq:second-derivative-of-s-tilde} by $\log p$ and subtract the result from the expression in \eqref{eq:g-tilde-double-prime}, we obtain
\begin{align*}
\tilde{g}''(p)&=(1-p)\{\log(1-p)-\log p\}r_2^{\tilde s}\log r_2 \cdot \frac{d^2\tilde{s}}{dp^2}\\
&\qquad +(1-p)\{\log(1-p)-\log p\}r_2^{\tilde s}(\log r_2)^2\cdot \left(\frac{d\tilde{s}}{dp}\right)^2\\
&\qquad +2\{r_1^{\tilde s}\log r_1-r_2^{\tilde s}\log r_2\}\cdot\frac{d\tilde{s}}{dp}\\
&\qquad -2\{\log(1-p)-\log p\}r_2^{\tilde s}\log r_2\cdot \frac{d\tilde{s}}{dp}\\
&\qquad +\frac{r_1^{\tilde s}}{p}+\frac{r_2^{\tilde s}}{1-p}\\
&=:A_1+A_2+A_3+A_4+A_5.
\end{align*}
Since $p\geq 1/2$ and $d\tilde{s}/dp\leq 0$, Lemmas \ref{lem:negative-difference} and \ref{lem:convex} imply that $A_1, A_3, A_4$ and $A_5$ are nonnegative; however, $A_2<0$. On the other hand, if we multiply \eqref{eq:second-derivative-of-s-tilde} by $\log(1-p)$ and subtract the result from the expression in \eqref{eq:g-tilde-double-prime}, we get
\begin{align*}
\tilde{g}''(p)&=p\{\log p-\log(1-p)\}r_1^{\tilde s}\log r_1\cdot \frac{d^2\tilde{s}}{dp^2}\\
&\qquad +p\{\log p-\log(1-p)\}r_1^{\tilde s}(\log r_1)^2\cdot \left(\frac{d\tilde{s}}{dp}\right)^2\\
&\qquad +2\{r_1^{\tilde s}\log r_1-r_2^{\tilde s}\log r_2\}\cdot\frac{d\tilde{s}}{dp}\\
&\qquad +2\{\log p-\log(1-p)\}r_1^{\tilde s}\log r_1\cdot \frac{d\tilde{s}}{dp}\\
&\qquad +\frac{r_1^{\tilde s}}{p}+\frac{r_2^{\tilde s}}{1-p}\\
&=:B_1+B_2+B_3+B_4+B_5.
\end{align*}
Here $B_2,\dots,B_5\geq 0$ but $B_1<0$. Note, however, that if
\begin{equation} \label{eq:lower-sufficient}
\frac{d^2\tilde{s}}{dp^2}\geq (-\log r_2)\left(\frac{d\tilde{s}}{dp}\right)^2,
\end{equation}
then $A_1+A_2\geq 0$; and if
\begin{equation} \label{eq:upper-sufficient}
\frac{d^2\tilde{s}}{dp^2}\leq (-\log r_1)\left(\frac{d\tilde{s}}{dp}\right)^2,
\end{equation}
then $B_1+B_2\geq 0$. Since $-\log r_1>-\log r_2$, at least one of \eqref{eq:lower-sufficient} and \eqref{eq:upper-sufficient} is the case. Hence, $\tilde{g}''(p)>0$ and so $\tilde{g}$ is strictly convex on $[1/2,1]$.

Furthermore, from \eqref{eq:g-tilde} it follows immediately that $\tilde{g}(1/2)=0$. Differentiating $\tilde{g}$ using \eqref{eq:g-tilde}, we get a form more suitable for substituting $p=1/2$, namely
\begin{align*}
\tilde{g}'(p)&=\left\{p\log(2p)\cdot r_1^{\tilde s(p)}\log r_1+(1-p)\log(2(1-p))\cdot r_2^{\tilde s(p)}\log r_2\right\}\frac{d\tilde{s}}{dp}\\
&\qquad +r_1^{\tilde s(p)}\big(\log(2p)+1\big)-r_2^{\tilde s(p)}\big(\log(2(1-p))+1\big),
\end{align*}
whence $\tilde{g}'(1/2)=r_1^{s_0}-r_2^{s_0}<0$, using also that $\tilde{s}(1/2)=s_0$.
Finally, from \eqref{eq:s-tilde-equation-bis} we have $r_1^{\tilde s(1)}=1/2$, which gives $\tilde{g}(1)=\frac12\log 2>0$. Hence, $\tilde{g}(p)$ has a unique additional zero $p^*$ in $(1/2,1)$, and $\tilde{g}(p)\leq 0$ if and only if $1/2\leq p\leq p^*$.
\end{proof}

\begin{proof}[Proof of Theorem \ref{cor:two-maps-nonhomogeneous}]
Assume $r_1<r_2$. Here inequality \eqref{eq:M-large-condition} becomes
\begin{equation*} 
g(p):=r_1^{s_0}\log(2p)+r_2^{s_0}\log(2(1-p))\geq 0.
\end{equation*}
Note that $g(p)\to -\infty$ as $p\searrow 0$, and $g(1/2)=0$. Furthermore,
\[
g'(p)=\frac{r_1^{s_0}}{p}-\frac{r_2^{s_0}}{1-p},
\]
so that $g'(1/2)<0$, and
\[
g''(p)=-\frac{r_1^{s_0}}{p^2}-\frac{r_2^{s_0}}{(1-p)^2}<0,
\]
which shows $g$ is strictly concave. Hence $g(p)$ has a unique zero $p_*$ in $(0,1/2)$, and $g(p)\geq 0$ if and only if $p_*\leq p\leq 1/2$.

Next, observe that \eqref{eq:M-small-condition} holds if and only if $\tilde{g}(p)\leq 0$, where $\tilde{g}$ is the function from Lemma \ref{lem:g-tilde}. By that lemma, $\tilde{g}$ has a unique zero $p^*$ in the interval $(1/2,1)$, and $\tilde{g}(p)\leq 0$ if and only if $1/2\leq p\leq p^*$.

In the third case (when $p\leq p_*$ or $p\geq p^*$), the formula for $\hat{s}$ follows by solving the linear system
\begin{align*}
h(x_1,x_2)&=-x_1\log r_1-x_2\log r_2=1,\\
\phi(x_1,x_2)&=x_1\log(2p)+x_2\log(2(1-p))=0,
\end{align*}
and substituting the result into $\psi$.
\end{proof}

\section*{Acknowledgments}
The first author is partially supported by Simons Foundation grant \# 709869.

\end{document}